\documentclass[preprint,12pt]{elsarticle}
\usepackage{mathrsfs}
\usepackage{dsfont}
\usepackage{amsmath, amssymb, amsthm, latexsym, amsfonts, epsfig, color,graphicx}

\numberwithin{equation}{section}

\renewcommand{\phi}{\varphi}
\newcommand{\e}{\epsilon}

\newtheorem{theorem}{Theorem}
\newtheorem{lemma}{Lemma}
\newtheorem{condition}{Condition}
\newtheorem{remark}{Remark}
\newtheorem{example}{Example}

\newtheorem{definition}{Definition}

    \newcommand\figcaption{\def\@captype{figure}\caption}
    \newcommand\tabcaption{\def\@captype{table}\caption}

\journal{Int. J. Nonlin. Mech.}

\begin{document}

\begin{frontmatter}
\title{Large deviations for stochastic nonlinear systems of\\
slow-fast diffusions with non-Gaussian L\'evy noises}

\author[a1]{Shenglan Yuan}
\author[a2]{Ren\'e Schilling}
\author[a3]{Jinqiao Duan}

\address[a1]{Institut f\"ur Mathematik, Universit\"at Augsburg,
86135 Augsburg, Germany\\e-mail: shenglan.yuan@math.uni-augsburg.de}
\address[a2]{Institut f\"{u}r Mathematische Stochastik in Technische Universit\"{a}t Dresden, Dresden, D-01069, Germany\\e-mail: rene.schilling@tu-dresden.de}
\address[a3]{Department of Applied Mathematics in Illinois Institute of Technology, Chicago, IL 60616, USA\\e-mail: duan@iit.edu}
\date{\today}
\begin{abstract}
We establish the large deviation principle for the slow variables in slow-fast dynamical system  driven by both Brownian noises and L\'evy noises. The fast variables evolve at much faster time scale than the slow variables, but they are fully inter-dependent. We study the asymptotics of
the logarithmic functionals of the slow variables in the three regimes based on viscosity solutions to the Cauchy problem for a sequence of partial integro-differential equations. We also verify the comparison principle for the related Cauchy problem to show the existence and uniqueness of the limit for viscosity solutions.
\end{abstract}
\end{frontmatter}

\section{Introduction}
\label{secton1}
Many nonlinear dynamical systems under random influences often involve the interplay of slow and fast variables. For instance, climate-weather interaction models, geophysical flows, macromolecules and planetary motion \cite{MW06,Pa01,Qi01}. The slow-fast systems described by stochastic differential equations are thought to be appropriate mathematical models for those perturbed nonlinear dynamical systems.

We consider the following nonlinear slow-fast stochastic dynamical system driven by both Brownian noises and L\'evy noises\,:
\begin{eqnarray}\label{1}
\left\{\begin{array}{l}
 dX^{\varepsilon,\delta}_{t}=\varepsilon b_{1}(X^{\varepsilon,\delta}_{t-},Y^{\varepsilon,\delta}_{t-})dt+\sqrt{2\varepsilon}\sigma_{1}(X^{\varepsilon,\delta}_{t-},Y^{\varepsilon,\delta}_{t-})dW_{t}^{(1)}\\
~~~~~~~~~~~~+\varepsilon\int_{\mathbb{R}\setminus \{0\}} k_{1}(X^{\varepsilon,\delta}_{t-},Y^{\varepsilon,\delta}_{t-},z)\tilde{N}^{(1),\frac{1}{\varepsilon}}(dz,dt),\quad X^{\varepsilon,\delta}_{0}=x_{0}\in\mathbb{R},  \\[0.5ex]
dY^{\varepsilon,\delta}_{t}=\frac{\varepsilon}{\delta}b_{2}(X^{\varepsilon,\delta}_{t-},Y^{\varepsilon,\delta}_{t-})dt+\sqrt{\frac{2\varepsilon}{\delta}}
\sigma_{2}(X^{\varepsilon,\delta}_{t-},Y^{\varepsilon,\delta}_{t-})(\rho dW_{t}^{(1)}+\sqrt{1-\rho^{2}}dW_{t}^{(2)})\\
~~~~~~~~~~~~+\int_{\mathbb{R}\setminus \{0\}} k_{2}(X^{\varepsilon,\delta}_{t-},Y^{\varepsilon,\delta}_{t-},z)\tilde{N}^{(2),\frac{\varepsilon}{\delta}}(dz,dt),\quad Y^{\varepsilon,\delta}_{0}=y_{0}\in\mathbb{R},
 \end{array}\right.
\end{eqnarray}
where $\tilde{N}^{(1),\frac{1}{\varepsilon}}(\cdot,\cdot), \tilde{N}^{(2),\frac{\varepsilon}{\delta}}(\cdot,\cdot)$ are independent compensated Poisson random measures
\[
\tilde{N}^{(1),\frac{1}{\varepsilon}}(\cdot,\cdot)=N^{(1),\frac{1}{\varepsilon}}(\cdot,\cdot)-\frac{1}{\varepsilon}\nu_{1}(dz)dt,\quad  \tilde{N}^{(2),\frac{\varepsilon}{\delta}}(\cdot,\cdot)=N^{(2),\frac{\varepsilon}{\delta}}(\cdot,\cdot)-\frac{\varepsilon}{\delta}\nu_{2}(dz)dt,
\]
with associated Poisson random measures $N^{(1),\frac{1}{\varepsilon}}(\cdot,\cdot), N^{(2),\frac{\varepsilon}{\delta}}(\cdot,\cdot)$ and intensity measures
$\frac{1}{\varepsilon}\nu_{1}(dz)dt$, $\frac{\varepsilon}{\delta}\nu_{2}(dz)dt$, in which $\nu_{i}, i=1,2$ are L\'evy measures, i.e., $\sigma$-finite measures on ${\mathbb{R}\setminus \{0\}}$ such that $\int_{\mathbb{R}\setminus \{0\}}(1\wedge z^{2})\nu_{i}(dz)<\infty$. $W^{(1)}, W^{(2)}$ are independent Brownian motions independent of $\tilde{N}^{(1),\frac{1}{\varepsilon}}(\cdot,\cdot), \tilde{N}^{(2),\frac{\varepsilon}{\delta}}(\cdot,\cdot)$, with $\rho\in(-1,1)$ constant. The two small positive parameters $\varepsilon, \delta$ describe the separation of time scale between the slow variable $X^{\varepsilon,\delta}_{t}$ and the fast variable $Y^{\varepsilon,\delta}_{t}$. Indeed, $Y^{\varepsilon,\delta}_{t}$ evolves at faster time scale $s=\frac{t}{\delta}$ than $X^{\varepsilon,\delta}_{t}$ with $\delta<\varepsilon\ll1$.

In the last years, the long time large deviations behavior of slow-fast systems has attracted more and more attention because of the various applications in the fields of statistical physics, engineering, chemistry and financial mathematics \cite{Ba03,DZ10,Ki04}. The behavior of the slow variables on time-scales that are much longer than that over which the fast variables evolve, can be characterized via a large deviation principle \cite{FW12}.

There exist fruitful results of large deviation principle for slow-fast systems based on Brownian noise \cite{Bo16,Ki09,Sp13,WRD}. Feng, Fouque and Kumar derived a large deviation principle for stochastic volatility models in two regimes where the maturity is small, and deduced asymptotic prices for out-of-the-money call and put options in \cite{FFK12}. Moreover, Bardi, Cesaroni and Ghilli \cite{BCG15} proved a large deviation principle for three regimes of stochastic systems affected by a stochastic volatility evolving at a faster time scale, and applied it to the asymptotics of option prices near maturity.

The study of the large deviation principle for slow-fast systems driven by non-Gaussian L\'evy noises is still in its infancy, but some interesting works are emerging \cite{KS12,YD19}. The large deviations for a specific class of slow-fast systems, where the slow process is a diffusion and the fast process is a mean-reverting process driven by a L\'evy process, was studied in \cite{BCS16}. For system of the form \eqref{1} with $\delta=\varepsilon^{2}$, in which the slow and fast jump-diffusions are fully inter-dependent, the slow process has small perturbative noise and the fast process is ergodic, a large deviation principle was established in \cite{KP17}. Their methods based on viscosity sloutions to the Cauchy problem for a sequence of partial integro-differential equations and a construction of the sub- and super-solutions to related Cauchy problems.

The viscosity solution theory is an appropriate tool to deal with various interesting partial integro-differential equations for which there exist no classical solutions \cite{AT16,BI08,Ph98}. In the viscosity method, the comparison principle has been used to prove the convergence of viscosity solutions to the Cauchy problem for partial integro-differential equations \cite{Aw91,Bc08,JK06}. To obtain the large deviation principle for system
\eqref{1}, we follow the viscosity solution approach in \cite{FK06,FFK12,DH08,KP17}.

The main goal of this work is to analyze in detail the structure of the large deviation principle for the slow variables $\{X_{t}^{\varepsilon}\}_{\varepsilon>0}$ of system \eqref{1} with $\delta=\varepsilon^{\alpha}, \alpha>1$  in three different regimes. Utilizing Bryc's inverse Varadhan lemma \cite[Section 4]{DZ10},
the key step is to show that the functionals $\{U^{\varepsilon}\}_{\varepsilon>0}$ satisfying the Cauchy problem \eqref{sU}
converge to some quantity independent of $y$ described by the Cauchy problem \eqref{sU0}.

We first take the relaxed upper and lower semilimits $U_{\uparrow}$ and $U_{\downarrow}$ of $\{U^{\varepsilon}\}_{\varepsilon>0}$ for the Cauchy problem \eqref{sU}, and then obtain the upper- and lower-semicontinuous functions $\hat{U}$ and $\check{U}$, respectively. Subsequently, by using an indexing set $\lambda\in\Lambda$, we construct a family of operators $\hat{H}^{\lambda}$ and $\check{H}^{\lambda}$,
 such that $\hat{U}$ is a subsolution of the Cauchy problem for the operator $\hat{H}=\inf_{\lambda\in\Lambda}\{\hat{H}^{\lambda}\}$, and $\check{U}$ is a supersolution to the Cauchy problem for operator $\check{H}=\sup_{\lambda\in\Lambda}\{\check{H}^{\lambda}\}$. After that we demonstrate a comparison principle between subsolution $\hat{U}$ and supersolution $\check{U}$. Finally, we check that this comparison principle implies convergence of solutions $\{U^{\varepsilon}\}_{\varepsilon>0}$ for the Cauchy problem \eqref{sU} with $H^{\varepsilon}$ on the compact subsets of $[0,T]\times\mathbb{R}\times\mathbb{R}$ to the unique viscosity solution $U^{0}$ for the Cauchy problem \eqref{sU0} with $H^{0}$.

This paper is organized as follows. In Section \ref{secton2}, we provide some precise conditions for the slow-fast system, and describe the Cauchy problem \eqref{sU} satisfied by $\{U^{\varepsilon}\}_{\varepsilon>0}$. In Section \ref{secton3}, we introduce the limit Hamiltonian $H^{0}$ that has different forms in the three regimes depending on $\alpha>1$: supercritical case for $\alpha>2$, critical case for $\alpha=2$ and subcritical case for $\alpha<2$. In Section \ref{secton4}, we derive the comparison principle and present the convergence result for solutions of the Cauchy problem \eqref{sU} with the  Hamiltonian $H^{\varepsilon}$ identified in Section \ref{secton2} to the unique viscosity solution of the Cauchy problem \eqref{sU0} with the limit Hamiltonian $H^{0}$ identified in Section \ref{secton3}. In Section \ref{secton5}, the large deviation principle for the slow variables $\{X_{t}^{\varepsilon}\}_{\varepsilon>0}$ is established, and its application is illustrated with an example. Finally, Section \ref{CFC} summarizes our findings as well as directions for future study.

\section{\bf Preliminaries}\label{secton2}
Throughout this paper, $(\Omega,\mathcal{F},\mathbb{P})$ is a probability space. Let $\mathcal{P}(\mathbb{R})$ denote the space of probability measures on $\mathbb{R}$. We consider Euclidean space $\mathbb{R}^{d}$ endowed with the Borel $\sigma$-algebra $\mathcal{B}(\mathbb{R}^{d})$. For a differentiable function $f: \mathbb{R}^{d}\rightarrow\mathbb{R}$,
the partial derivative with respect to $x$ is denoted by $\partial_{x}f$. As usual,
$C_{b}^{k}(\mathbb{R}^{d})$ is the space of $k$-times bounded continuously differentiable functions, $C_{b}(\mathbb{R}^{d})$ is the space of bounded uniformly continuous functions, and $C_{c}(\mathbb{R}^{d})$ is the space of the continuous functions with compact support. And we use $``:="$ as a way of definition.

To keep notation as simple as possible, we restrict ourselves to one-dimensional case. The variables $X^{\varepsilon,\delta}_{t}$ and $Y^{\varepsilon,\delta}_{t}$  in system \eqref{1} lie in Euclidean state space $\mathbb{R}$ that is locally compact.
Most of the results for multi-dimensional case can be proved in a similar fashion by considering the coordinates.

We assume the following conditions since both stochastic terms enjoy It\^{o}'s isometry. The nonlinear functions $ b_{1}(x,y)$, $b_{2}(x,y)$, $\sigma_{1}(x,y)$, $\sigma_{2}(x,y)$, $k_{1}(x,y,z)$, $k_{2}(x,y,z)$ in system \eqref{1} satisfy
\begin{description}
  \item[$(\textbf{C}_{1})$ \textbf{Lipschitz condition}\,: ] $\exists\, K_{1}>0$ such that, for all $(x_{1},y_{1}), (x_{2},y_{2})\in\mathbb{R}^{2}$,
\begin{align*}
  &|b_{1}(x_{2},y_{2})-b_{1}(x_{1},y_{1})|^{2}+|b_{2}(x_{2},y_{2})-b_{2}(x_{1},y_{1})|^{2}\\
  &+|\sigma_{1}(x_{2},y_{2})-\sigma_{1}(x_{1},y_{1})|^{2}+\int_{\mathbb{R}\setminus \{0\}}|k_{1}(x_{2},y_{2},z)-k_{1}(x_{1},y_{1},z)|^{2}\nu_{1}(dz)\\
  &+|\sigma_{2}(x_{2},y_{2})-\sigma_{2}(x_{1},y_{1})|^{2}+\int_{\mathbb{R}\setminus \{0\}} |k_{2}(x_{2},y_{2},z)-k_{2}(x_{1},y_{1},z)|^{2}\nu_{2}(dz)\\
  &\leq K_{1}[\,(x_{2}-x_{1})^{2}+(y_{2}-y_{1})^{2}\,].
\end{align*}
  \item[$(\textbf{C}_{2})$ \textbf{Growth condition}\,:]  $\exists\, K_{2}>0$ such that, for all $(x,y)\in\mathbb{R}^{2}$,
\begin{align*}
  &|b_{1}(x,y)|^{2}+|b_{2}(x,y)|^{2}+|\sigma_{1}(x,y)|^{2}+|\sigma_{2}(x,y)|^{2}\\
  &+\int_{\mathbb{R}\setminus \{0\}}|k_{1}(x,y,z)|^{2}\nu_{1}(dz)+\int_{\mathbb{R}\setminus \{0\}}|k_{2}(x,y,z)|^{2}\nu_{2}(dz)\\
  &\leq K_{2}(1+x^{2}+y^{2}).
\end{align*}
\end{description}
Conditions ($\textbf{C}_{1}$) and ($\textbf{C}_{2}$) ensure that system \eqref{1} has a unique strong solution, and which is a Markov process; see \cite[Chapter 6]{Ap04}.  Moreover, if the L\'evy measures $\nu_{1},\nu_{2}$ are finite, the growth condition ($\textbf{C}_{2}$) is a consequence of the Lipschitz condition ($\textbf{C}_{1}$).
\begin{remark}
We imposed the conditions ($\textbf{\rm C}_{1}$) and ($\textbf{\rm C}_{2}$) on the mappings $ b_{1}$, $b_{2}$, $\sigma_{1}$, $\sigma_{2}$, $k_{1}$ and $k_{2}$ to guarantee that c\`{a}dl\`{a}g solutions of system \eqref{1} exist. If we take $k_1=k_2=0$,  the Lipschitz condition ($\textbf{\rm C}_{1}$) enables
us to solve system \eqref{1}. Hence in the case of non-zero $k_{1}$ and $k_{2}$, in the presence
of ($\textbf{\rm C}_{1}$), the growth condition ($\textbf{\rm C}_{2}$) is equivalent to the requirement that there exists $K_{2}>0$ such that,
for all $(x,y)\in\mathbb{R}^{2}$,
\begin{equation*}
\int_{\mathbb{R}\setminus \{0\}}|k_{1}(x,y,z)|^{2}\nu_{1}(dz)+\int_{\mathbb{R}\setminus \{0\}}|k_{2}(x,y,z)|^{2}\nu_{2}(dz)\leq K_{2}(1+x^{2}+y^{2}).
\end{equation*}
\end{remark}

For each $f\in C_{b}^{2}(\mathbb{R}^{2})$, the infinitesimal generator $\mathcal{L}^{\varepsilon,\delta}$ of the solution $(X^{\varepsilon,\delta},Y^{\varepsilon,\delta})$ for system \eqref{1} is
{\small\begin{align}\nonumber
\mathcal{L}^{\varepsilon,\delta}f(x,y)=&\varepsilon\left(b_{1}(x,y)\partial_{x}f(x,y)+\sigma_{1}^{2}(x,y)\partial_{xx}^{2}f(x,y)\right)+\frac{2\varepsilon}{\sqrt{\delta}}\rho\sigma_{1}(x,y)\sigma_{2}(x,y)\partial_{xy}^{2}f(x,y)\\ \nonumber
  &+\frac{1}{\varepsilon}\int_{\mathbb{R}\setminus \{0\}}\big(f(x+\varepsilon k_{1}(x,y,z),y)-f(x,y)-\varepsilon k_{1}(x,y,z)\partial_{x}f(x,y)\big)\nu_{1}(dz)\\ \nonumber
  &+\frac{\varepsilon}{\delta}\Big[\,b_{2}(x,y)\partial_{y}f(x,y)+\sigma_{2}^{2}(x,y)\partial_{yy}^{2}f(x,y)\\ \label{LL}
  &+\int_{\mathbb{R}\setminus \{0\}}\big(f(x,y+k_{2}(x,y,z))-f(x,y)-k_{2}(x,y,z)\partial_{y}f(x,y)\big)\nu_{2}(dz)\,\Big].
\end{align}}

In order to understand the two-scale $\varepsilon, \delta\rightarrow0$ limit behaviors of the slow variables $X_{t}^{\varepsilon,\delta}$, we introduce the virtual fast process $Y^{x}$, which satisfies
\begin{align}\nonumber
dY_{t}^{x}&=b_{2}(x,Y_{t-}^{x})dt+\sqrt{2}
\sigma_{2}(x,Y_{t-}^{x})\Big(\rho dW_{t}^{(1)}+\sqrt{1-\rho^{2}}dW_{t}^{(2)}\Big)\\ \label{YYx}
&~~~~~+\int_{\mathbb{R}\setminus \{0\}}k_{2}(x,Y_{t-}^{x},z)\tilde{N}^{(2)}(dz,dt),\quad Y_{0}^{x}=y_{0}\in\mathbb{R},
\end{align}
where $x$ is fixed. The above equation is originated from equation for $Y_{t}^{\varepsilon,\delta}$ in system \eqref{1} by setting $X_{t}^{\varepsilon,\delta}$ to $x$ and rescaling time $t\rightarrow\frac{\delta}{\varepsilon}t$. The infinitesimal generator $\mathcal{L}^{x}$ of $Y^{x}$ is
\begin{align}\nonumber
\mathcal{L}^{x}f(y)=&b_{2}(x,y)\partial_{y}f(y)+\sigma^{2}_{2}(x,y)\partial_{yy}^{2}f(y) \\ \label{LLx}
&+\int_{\mathbb{R}\setminus \{0\}}\big(f(y+k_{2}(x,y,z))-f(y)-k_{2}(x,y,z)\partial_{y}f(y)\big)\nu_{2}(dz),
\end{align}
where $f\in C_{b}^{2}(\mathbb{R})$. For fixed $x, p\in\mathbb{R}$, define the following perturbed generator $\mathcal{L}^{x,p}$\,:
\begin{equation}\label{LLp}
\mathcal{L}^{x,p}f(y)=\mathcal{L}^{x}f(y)+2\rho\sigma_{1}(x,y)\sigma_{2}(x,y)\partial_{y}f(y)p,
\end{equation}
and let $Y^{x,p}$ be the process corresponding to the generator $\mathcal{L}^{x,p}$ as in \eqref{LLp}.

If there is no jump term in the right side of \eqref{YYx}, i.e., without the intergral term, then the equation
\begin{equation}\label{tilde}
d\tilde{Y}_{t}^{x}=b_{2}(x,\tilde{Y}_{t}^{x})dt+\sqrt{2}\sigma_{2}(x,\tilde{Y}_{t}^{x})\Big(\rho dW_{t}^{(1)}+\sqrt{1-\rho^{2}}dW_{t}^{(2)}\Big),\quad \tilde{Y}_{0}^{x}=y_{0}\in\mathbb{R},
\end{equation}
defines a Markov process $\tilde{Y}^{x}$ with generator $\mathcal{\tilde{L}}^{x}$ as follow\,:
\begin{equation}\label{lx}
\mathcal{\tilde{L}}^{x}f(y)=b_{2}(x,y)\partial_{y}f(y)+\sigma^{2}_{2}(x,y)\partial_{yy}^{2}f(y),\quad\text{for}~f\in C_{b}^{2}(\mathbb{R}).
\end{equation}
The scale and speed measures of the process $\tilde{Y}^{x}$ in \eqref{tilde} are given by
$$
s(y):=\exp\{-\int_{-\infty}^{y}\frac{b_{2}(x,r)}{\sigma^{2}_{2}(x,r)}dr\},\,~~\, m(y):=\frac{1}{\sigma^{2}_{2}(x,y)s(y)}.
$$
Denoting $dS(y):=s(y)dy$ and $dM(y):=m(y)dy$, we have
$$
\mathcal{\tilde{L}}^{x}f(y)=\frac{d}{dM}\left(\frac{df(y)}{dS}\right).
$$
There exists a unique probability measure
\begin{equation}\label{piy}
\pi(dy):=\frac{m(y)}{\int_{\mathbb{R}}m(y)dy}dy
\end{equation}
such that $\int_{\mathbb{R}}\mathcal{\tilde{L}}^{x}f(y)\pi(dy)=0$; see \cite[Chapter 15]{KT81}.

Let $g\in C_{b}(\mathbb{R})$ and define the following functionals\,:
\begin{equation}\label{VV}
V^{\varepsilon,\delta}(t,x,y):=\mathbb{E}\big[\,g(X_{t}^{\varepsilon,\delta})\,|\,X_{0}^{\varepsilon,\delta}=x,Y_{0}^{\varepsilon,\delta}=y\,\big].
\end{equation}
In general, $V^{\varepsilon,\delta}\in C_{b}([0,T]\times\mathbb{R}\times\mathbb{R})$. Moreover, if $V^{\varepsilon,\delta}\in C^{1,2}([0,T]\times\mathbb{R}\times\mathbb{R})$, then $V^{\varepsilon,\delta}$ solve the following Cauchy problem in the classical sense\,:
\begin{eqnarray}\label{VV}
\left\{\begin{array}{l}
\partial_{t}V(t,x,y)=\mathcal{L}^{\varepsilon,\delta}V(t,x,y),\quad(t,x,y)\in(0,T]\times\mathbb{R}\times\mathbb{R};\\
~~V(0,x,y)=g(x),\quad(x,y)\in\mathbb{R}\times\mathbb{R},
 \end{array}\right.
\end{eqnarray}
where $\mathcal{L}^{\varepsilon,\delta}$ as in \eqref{LL}. When $g(x)=e^{\frac{h(x)}{\varepsilon}}$ with $h\in C_{b}(\mathbb{R})$,
we gain
\begin{equation}\label{VVeps}
V^{\varepsilon,\delta}(t,x,y)= \mathbb{E}\big[\,e^{\frac{h(X_{t}^{\varepsilon,\delta})}{\varepsilon}}\,|\,X_{0}^{\varepsilon,\delta}=x,Y_{0}^{\varepsilon,\delta}=y\,\big].
\end{equation}
Using the logarithmic transform method in \cite{FK06,FS06}, define
\begin{equation}\label{UU}
U^{\varepsilon,\delta}(t,x,y)=\varepsilon\ln V^{\varepsilon,\delta}(t,x,y),
\end{equation}
where $V^{\varepsilon,\delta}$ are taken from \eqref{VVeps}. Inserting $V^{\varepsilon,\delta}(t,x,y)=e^{\frac{U^{\varepsilon,\delta}(t,x,y)}{\varepsilon}}$ into \eqref{VV}, at least informally, \eqref{sU} below is satisfied. In the absence of knowledge on smoothness of $V^{\varepsilon,\delta}$, we can only conclude that $U^{\varepsilon,\delta}$ tackle the Cauchy problem \eqref{sU} in the sense of viscosity solution (Definition \ref{vis}).

\begin{lemma} \label{HUhL}
For each $h\in C_{b}(\mathbb{R})$ depending only on the variable $x$, $U^{\varepsilon,\delta}(t,x,y)$ defined by \eqref{UU} is a viscosity solution of the Cauchy problem\,:
\begin{eqnarray}\label{sU}
\left\{\begin{array}{l}
\partial_{t}U(t,x,y)=H^{\varepsilon,\delta}U(t,x,y),\quad(t,x,y)\in(0,T]\times\mathbb{R}\times\mathbb{R};\\
~~U(0,x,y)=h(x),\quad(x,y)\in\mathbb{R}\times\mathbb{R}.
 \end{array}\right.
\end{eqnarray}
In the above, the nonlinear nonlocal operator $H^{\varepsilon,\delta}$ is the exponential generator\,:
\begin{align} \nonumber
H^{\varepsilon,\delta}U(t,x,y)&=\varepsilon e^{-\frac{U(t,x,y)}{\varepsilon}}\mathcal{L}^{\varepsilon,\delta}e^{\frac{U(t,x,y)}{\varepsilon}}\\ \nonumber
&=\varepsilon\left(b_{1}(x,y)\partial_{x}U(t,x,y)+\sigma_{1}^{2}(x,y)\partial_{xx}^{2}U(t,x,y)\right) +\sigma_{1}^{2}(x,y)(\partial_{x}U(t,x,y))^{2}\\[1mm] \nonumber
&~~~+2\rho\sigma_{1}(x,y)\sigma_{2}(x,y)\left(\frac{1}{\sqrt{\delta}}\partial_{x}U(t,x,y)\partial_{y}U(t,x,y)+\frac{\varepsilon}{\sqrt{\delta}}\partial_{xy}^{2}U(t,x,y)\right)\\[1mm] \nonumber
&~~~+\int_{\mathbb{R}\setminus \{0\}}\Big(e^{\frac{1}{\varepsilon}[U(t,x+\varepsilon k_{1}(x,y,z),y)-U(t,x,y)]}-1-k_{1}(x,y,z)\partial_{x}U(t,x,y)\Big)\nu_{1}(dz)\\ \label{Hed}
&~~~+\frac{\varepsilon^{2}}{\delta}e^{-\frac{U(t,x,y)}{\varepsilon}}\mathcal{L}^{x}e^{\frac{U(t,x,y)}{\varepsilon}},
\end{align}
where $\mathcal{L}^{x}$ defined as \eqref{LLx}.
\end{lemma}
\begin{remark}
Note that $H^{\varepsilon,\delta}$ only operates on the spatial variables $x$ and $y$.  By utilizing similar arguments as in \cite{FK06,FFK12,KP17}, We could certify for Lemma \ref{HUhL}.
\end{remark}

We want to study the large deviation behaviors of the slow variables $X_{t}^{\varepsilon,\delta}$ in system \eqref{1} as both $\varepsilon$ and $\delta$ go to $0$, and we expect different limits behaviors depending on the ratio $\frac{\varepsilon}{\delta}$. Therefore we put $\delta=\varepsilon^{\alpha}, \alpha>1$, and denote the variables $X_{t}^{\varepsilon,\delta}$ and $Y_{t}^{\varepsilon,\delta}$ by $X_{t}^{\varepsilon}$ and $Y_{t}^{\varepsilon}$. Thus, the system \eqref{1} can be rewritten into
{\small\begin{eqnarray}\label{2}
\left\{\begin{array}{l}
 dX_{t}^{\varepsilon}=\varepsilon b_{1}(X_{t-}^{\varepsilon},Y_{t-}^{\varepsilon})dt+\sqrt{2\varepsilon}\sigma_{1}(X_{t-}^{\varepsilon},Y_{t-}^{\varepsilon})dW_{t}^{(1)}\\~~~~~~~~~+\varepsilon\int_{\mathbb{R}\setminus \{0\}}k_{1}(X_{t-}^{\varepsilon},Y_{t-}^{\varepsilon},z)\tilde{N}^{(1),\frac{1}{\varepsilon}}(dz,dt),\quad X_{0}^{\varepsilon}=x_{0}\in\mathbb{R},  \\[1ex]
dY_{t}^{\varepsilon}=\varepsilon^{1-\alpha}b_{2}(X_{t-}^{\varepsilon},Y_{t-}^{\varepsilon})dt+\sqrt{2\varepsilon^{1-\alpha}}
\sigma_{2}(X_{t-}^{\varepsilon},Y_{t-}^{\varepsilon})(\,\rho dW_{t}^{(1)}+\sqrt{1-\rho^{2}}dW_{t}^{(2)})\\~~~~~~~~~+\int_{\mathbb{R}\setminus \{0\}} k_{2}(X_{t-}^{\varepsilon},Y_{t-}^{\varepsilon},z)\tilde{N}^{(2),\varepsilon^{1-\alpha}}(dz,dt),\quad Y_{0}^{\varepsilon}=y_{0}\in\mathbb{R}.
 \end{array}\right.
\end{eqnarray}}
Hence, for notational simplicity, we drop the subscript $\delta$, and write $U^{\varepsilon}$ and $H^{\varepsilon}$ for $U^{\varepsilon,\delta}$ and $H^{\varepsilon,\delta}$, respectively.

For each $x$ and $p$ in $\mathbb{R}$, define
{\small\begin{equation}\label{v}
V^{x,p}(y):=\sigma_{1}^{2}(x,y)p^{2}+\int_{\mathbb{R}\setminus \{0\}}\Big(e^{k_{1}(x,y,z)p}-1-k_{1}(x,y,z)p\Big)\nu_{1}(dz).
\end{equation}}
We suppose that $V^{x,p}$ is bounded from below, i.e., there exists $C^{x,p}>-\infty$ such that
\begin{equation}\label{VK}
V^{x,p}(y)\geq C^{x,p},~\text{for any}\,\,  y\in\mathbb{R}.
\end{equation}
In the following we will assume other three system conditions\,:
\begin{description}
\item[$(\textbf{C}_{3})$\textbf{ Periodic condition}\,:] The functions $b_{1}(x,y)$, $b_{2}(x,y)$, $\sigma_{1}(x,y)$, $\sigma_{2}(x,y)$, $k_{1}(x,y,z)$, $k_{2}(x,y,z)$ in system \eqref{1} are periodic with respect to the variable $y$.
  \item[$(\textbf{C}_{4})$\textbf{ Ergodicity condition}\,:] The perturbed fast process $Y^{x,p}$ with generator $\mathcal{L}^{x,p}$ in \eqref{LLp} is ergodic at every $x$ with respect to it's unique invariant measure.
  \item[$(\textbf{C}_{5})$ \textbf{Lyapunov condition}\,:]  For a positive function $\xi(\cdot)\in C^{2}(\mathbb{R})$ such that $\xi(\cdot)$ has a compact finite level set, and for each $\theta\in(0,1]$ and compact set $K\subset\mathbb{R}$,
      \begin{description}
        \item[$\textbf{(i)}$\,:]  $\forall\,l^{i}\in\mathbb{R}, \forall\,x\in\mathbb{R}, \forall\,p\in K$, there exists a compact set $J^{i}_{c,\theta,K}\subset\mathbb{R}$ such that
\begin{align*}
&  \Big\{y\in\mathbb{R}:-\theta\big(2\varepsilon^{\frac{\alpha}{2}-1}\rho\sigma_{1}(x,y)\sigma_{2}(x,y)\partial_{y}\xi(y)p+\varepsilon^{2-\alpha}e^{-\varepsilon^{\alpha-2}\xi(y)}\mathcal{L}^{x}e^{\varepsilon^{\alpha-2}\xi(y)}\big) \\
  &-\big(|V^{x,p}(y)|+|b_{1}(x,y)p|+\sigma_{1}^{2}(x,y)\big)\leq l^{i}\Big\}\subset J^{i}_{c,\theta,K}.
\end{align*}
        \item[$\textbf{(ii)}$\,:]  $\forall\,l^{ii}\in\mathbb{R}, \forall\,x\in\mathbb{R}, \forall\,p\in K$, there exists a compact set $J^{ii}_{c,\theta,K}\subset\mathbb{R}$ such that
\begin{equation*}
\Big\{y\in\mathbb{R}:-\theta e^{-\xi(y)}\mathcal{L}^{x,p}e^{\xi(y)}-\big(|V^{x,p}(y)|+|b_{1}(x,y)p|+\sigma_{1}^{2}(x,y)\big)\leq l^{ii}\Big\}\subset J^{ii}_{c,\theta,K}.
\end{equation*}
        \item[$\textbf{(iii)}$\,:] $\forall\,l^{iii}\in\mathbb{R}, \forall\,x\in\mathbb{R}, \forall\,p\in K$, there exists a compact set $J^{i}_{c,\theta,K}\subset\mathbb{R}$ such that
\begin{align*}
&  \Big\{y\in\mathbb{R}:-\theta\big(2\rho\sigma_{1}(x,y)\sigma_{2}(x,y)\partial_{y}\xi(y)p+\varepsilon^{2-\alpha}e^{-\varepsilon^{\frac{\alpha}{2}-1}\xi(y)}\mathcal{L}^{x}e^{\varepsilon^{\frac{\alpha}{2}-1}\xi(y)}\big) \\
  &-\big(|V^{x,p}(y)|+|b_{1}(x,y)p|+\sigma_{1}^{2}(x,y)\big)\leq l^{iii}\Big\}\subset J^{iii}_{c,\theta,K}.
\end{align*}
      \end{description}
\end{description}

\section{\bf Limit Hamiltonian $H^{0}$}\label{secton3}
 Our goal is to study the limit of the functionals $\{U^{\varepsilon}\}_{\varepsilon>0}$ described in \eqref{sU} as $\varepsilon\rightarrow0$. Following the viscosity solution approach for the Cauchy problem of partial integro-differential equations (see \cite{JH16}), we need to identify a suitable limit Hamiltonian $H^{0}$, and characterize the limit of $\{U^{\varepsilon}\}_{\varepsilon>0}$ as the unique viscosity solution of an appropriate Cauchy problem with the limit Hamiltonian $H^{0}$.

 Below, we will use formal asymptotic expansions tools to discover the limit Hamiltonian $H^{0}$, which has different forms in the three regimes depending on $\alpha>1$: supercritical case for $\alpha>2$, critical case for $\alpha=2$ and subcritical case for $\alpha<2$. Those formal derivations are complementary to the rigorous ones \cite{BCG15,Bo16,Ki04,Li96}.
\begin{description}
\item[\textbf{ The supercritical case\,: $\alpha>2$}.]  Substituting the asymptotic expansion
\begin{equation}\label{sion}
U^{\varepsilon}(t,x,y)=U^{0}(t,x)+\varepsilon^{\alpha-1}W(t,x,y)
\end{equation}
into the first equation in \eqref{sU}, and collecting terms of $O(1)$ in $\varepsilon$, we get
\begin{align} \nonumber
\partial_{t}U^{0}(t,x)&=\sigma_{1}^{2}(x,y)\big(\partial_{x}U^{0}(t,x)\big)^{2}+b_{2}(x,y)\partial_{y}W(t,x,y)+\sigma_{2}^{2}(x,y)\partial_{yy}^{2}W(t,x,y)\\ \nonumber
&~~~+\int_{\mathbb{R}\setminus \{0\}}\Big(e^{k_{1}(x,y,z)\partial_{x}U^{0}(t,x)}-1-k_{1}(x,y,z)\partial_{x}U^{0}(t,x)\Big)\nu_{1}(dz)\\ \label{superU0}
&={\tilde{\mathcal{L}}}^{x}W(t,x,y)+V^{x, \partial_{x}U^{0}(t,x)}(y).
\end{align}
That is,
\begin{equation}\label{superU00}
{\tilde{\mathcal{L}}}^{x}W(t,x,y)=\partial_{t}U^{0}(t,x)-V^{x, \partial_{x}U^{0}(t,x)}(y),
\end{equation}
where both $U^{0}$ and $W$ are assumed to be independent of $\varepsilon$, and $V^{x, \partial_{x}U^{0}(t,x)}(y)$ as in \eqref{v} with $p=\partial_{x}U^{0}(t,x)$. The equation \eqref{superU00} has a unique solution $W$ with respect to the operator ${\tilde{\mathcal{L}}}^{x}$ from \eqref{lx} in the $y$ variable. Moreover,
\begin{equation}\label{super0U0}
\partial_{t}U^{0}(t,x)=\int_{\mathbb{R}}V^{x, \partial_{x}U^{0}(t,x)}(y)\pi(dy):=H^{0}(x, \partial_{x}U^{0}(t,x)),
\end{equation}
where $\pi(dy)$ as defined in \eqref{piy}.
  \item[\textbf{The critical case\,: $\alpha=2$}.] The first equation in \eqref{sU} with $\alpha=2$ becomes
{\small\begin{align} \nonumber
\partial_{t}U(t,x,y)&=\varepsilon\left(b_{1}(x,y)\partial_{x}U(t,x,y)+\sigma_{1}^{2}(x,y)\partial_{xx}^{2}U(t,x,y)\right)+\sigma_{1}^{2}(x,y)\big(\partial_{x}U(t,x,y)\big)^{2}\\[1mm] \nonumber
&~~~+2\rho\sigma_{1}(x,y)\sigma_{2}(x,y)\left(\varepsilon^{-1}\partial_{x}U(t,x,y)\partial_{y}U(t,x,y)+\partial_{xy}^{2}U(t,x,y)\right)\\[1mm] \nonumber
&~~~+\varepsilon^{-1}b_{2}(x,y)\partial_{y}U(t,x,y)+\varepsilon^{-2}\sigma_{2}^{2}(x,y)\big(\partial_{y}U(t,x,y)\big)^{2}+\varepsilon^{-1}\sigma_{2}^{2}(x,y)\partial_{yy}^{2}U(t,x,y)\\[1mm] \nonumber
&~~~+\int_{\mathbb{R}\setminus \{0\}}\Big(e^{\varepsilon^{-1}[\,U(t,x+\varepsilon k_{1}(x,y,z),y)-U(t,x,y)\,]}-1-k_{1}(x,y,z)\partial_{x}U(t,x,y)\Big)\nu_{1}(dz)\\ \label{6}
&~~~+\int_{\mathbb{R}\setminus \{0\}}\Big(e^{\varepsilon^{-1}[\,U(t,x,y+k_{2}(x,y,z))-U(t,x,y)\,]}-1-\varepsilon^{-1}k_{2}(x,y,z)\partial_{y}U(t,x,y)\Big)\nu_{2}(dz).
\end{align}}
We plug in the equation \eqref{6} the asymptotic expansion
\begin{equation}\label{sion1}
U^{\varepsilon}(t,x,y)=U^{0}(t,x)+\varepsilon W(t,x,y),
\end{equation}
and collect terms that are $O(1)$ in $\varepsilon$, we obtain
{\small\begin{align} \nonumber
\partial_{t}U^{0}(t,x)&=\sigma_{1}^{2}(x,y)\big(\partial_{x}U^{0}(t,x)\big)^{2}+2\rho\sigma_{1}(x,y)\sigma_{2}(x,y)\partial_{x}U^{0}(t,x)\partial_{y}W(t,x,y)\\[1mm] \nonumber
&~~~+b_{2}(x,y)\partial_{y}W(t,x,y)+\sigma_{2}^{2}(x,y)\big(\partial_{y}W(t,x,y)\big)^{2}+\sigma_{2}^{2}(x,y)\partial_{yy}^{2}W(t,x,y)\\[1mm] \nonumber
&~~~+\int_{\mathbb{R}\setminus \{0\}}\Big(e^{k_{1}(x,y,z)\partial_{x}U^{0}(t,x)}-1-k_{1}(x,y,z)\partial_{x}U^{0}(t,x)\Big)\nu_{1}(dz)\\ \nonumber
&~~~+\int_{\mathbb{R}\setminus \{0\}}\Big(e^{W(t,x,y+k_{2}(x,y,z))-W(t,x,y)}-1-k_{2}(x,y,z)\partial_{y}W(t,x,y)\Big)\nu_{2}(dz)\\\label{criU0}
&=e^{-W(t,x,y)}\mathcal{L}^{x, \partial_{x}U^{0}(t,x)}e^{W(t,x,y)}+V^{x, \partial_{x}U^{0}(t,x)}(y).
\end{align}}
Denote $\partial_{t}U^{0}(t,x)$ by $\lambda$ and $\partial_{x}U^{0}(t,x)$ by $p$. Fix $t, x$ and hence $\lambda, p$. The equation $\eqref{criU0}$ can be rewritten as
\begin{equation}\label{7}
e^{-W(t,x,y)}\mathcal{L}^{x,p}e^{W(t,x,y)}+V^{x,p}(y)=\lambda,
\end{equation}
where $\mathcal{L}^{x,p}$ and $V^{x,p}(y)$ are defined by \eqref{LLp} and \eqref{v} respectively. Multiplying both sides of \eqref{7} by $e^{W(t,x,y)}$, we get the following eigenvalues problem:
\begin{equation}\label{9}
\big(\mathcal{L}^{x,p}+V^{x,p}(y)\big)e^{W(t,x,y)}=\lambda e^{W(t,x,y)}.
\end{equation}
Note that the eigenvalue $\lambda$ depends on $x$ and $p$ due to \eqref{7}. Let $H^{0}(x,p):=\lambda$, then by \eqref{criU0}, \eqref{7} and $p=\partial_{x}U^{0}(t,x)$, we get
\begin{equation}\label{cri0U0}
\partial_{t}U^{0}(t,x)=H^{0}(x, \partial_{x}U^{0}(t,x)).
\end{equation}
We will show (in Section \ref{CR}) rigorously that the limiting operator $H^{0}$ is the principal eigenvalue $\lambda$ of the operator $\mathcal{L}^{x,p}+V^{x,p}(y)$ with eigenfunction $e^{W(t,x,y)}$. In order to obtain the principal eigenvalue $H^{0}(x,p)$, we use a Donsker-Varadhan variational representation as in \cite{DV75}. It follows that the principal eigenvalue $H^{0}(x,p)$ of $\mathcal{L}^{x,p}+V^{x,p}(y)$ is given by
\begin{equation}\label{hh0}
H^{0}(x,p)=\sup_{\mu\in\mathcal{P}(\mathbb{R})}\big[\int_{\mathbb{R}}V^{x,p}(y)d\mu(y)-J^{x,p}(\mu)\big].
\end{equation}
Here $V^{x,p}(y)$ stems from \eqref{v}, and the rate function $J^{x,p}(.):\mathcal{P}(\mathbb{R})\rightarrow\mathbb{R}\cup\{+\infty\}$ is defined by
\begin{equation}\label{j}
J^{x,p}(\mu):=-\inf\limits_{f\in D^{++}(\mathcal{L}^{x,p})}\int_{\mathbb{R}}\frac{\mathcal{L}^{x,p}f(y)}{f(y)}d\mu(y),
\end{equation}
where $D^{++}(\mathcal{L}^{x,p})\subset C_{b}(\mathbb{R})$ denotes the domain of $\mathcal{L}^{x,p}$ with functions that are strictly bounded below by a positive constant. Finally note that $H^{0}(x,p)$ in \eqref{hh0} is convex.
  \item[\textbf{The subcritical case\,: $\alpha<2$.}] Plugging the asymptotic expansion
\begin{equation}\label{sion2}
U^{\varepsilon}(t,x,y)=U^{0}(t,x)+\varepsilon^{\frac{\alpha}{2}}W(t,x,y)
\end{equation}
in the first equation in \eqref{sU}, we gain
\begin{align} \nonumber
\partial_{t}U^{0}(t,x)&=\sigma_{1}^{2}(x,y)\big(\partial_{x}U^{0}(t,x)\big)^{2}+2\rho\sigma_{1}(x,y)\sigma_{2}(x,y)\partial_{x}U^{0}(t,x)\partial_{y}W(t,x,y)\\ \nonumber
                      &~~~+\big[\sigma_{2}^{2}(x,y)+\frac{1}{2}\int_{\mathbb{R}\setminus\{0\}}k_{2}^{2}(x,y,z)\nu_{2}(dz)\big]\big(\partial_{y}W(t,x,y)\big)^{2}\\ \label{subU0}
                      &~~~+\int_{\mathbb{R}\setminus \{0\}}\Big(e^{k_{1}(x,y,z)\partial_{x}U^{0}(t,x)}-1-k_{1}(x,y,z)\partial_{x}U^{0}(t,x)\Big)\nu_{1}(dz).
\end{align}
We want to eliminate $W$ and the dependence on $y$ in \eqref{subU0}, and remain with the right hand side of the form $H^{0}(x,p)$ with $p=\partial_{x}U^{0}(t,x)$. Denote $\delta(x,y):=\frac{1}{2}\int_{\mathbb{R}\setminus\{0\}}k_{2}^{2}(x,y,z)\nu_{2}(dz)$, then from
{\small\begin{equation*}
H^{0}(x,p)=V^{x,p}(y)+2\rho\sigma_{1}(x,y)\sigma_{2}(x,y)p\partial_{y}W(t,x,y)+[\sigma_{2}^{2}(x,y)+\delta(x,y)](\partial_{y}W(t,x,y))^{2},
\end{equation*}}
we get
$$
\partial_{y}W(t,x,y)=\frac{\sqrt{H^{0}(x,p)-V^{x,p}(y)+\frac{\rho^{2}\sigma_{1}^{2}(x,y)\sigma_{2}^{2}(x,y)p^{2}}{\sigma_{2}^{2}(x,y)+\delta(x,y)}}-\frac{\rho\sigma_{1}(x,y)\sigma_{2}(x,y)p}{\sqrt{\sigma_{2}^{2}(x,y)+\delta(x,y)}}}{\sqrt{\sigma_{2}^{2}(x,y)+\delta(x,y)}}.
$$
It follows from $(\textbf{C}_{3})$ that $U^{\varepsilon}(t,x,y)$ is periodic with respect to the variable $y$, and then $W(t,x,y)$ is also periodic in the $y$ variable. If $H^{0}(x,p)\geq V^{x,p}(y)$, we have
$$
\sqrt{H^{0}(x,p)-V^{x,p}(y)+\frac{\rho^{2}\sigma_{1}^{2}(x,y)\sigma_{2}^{2}(x,y)p^{2}}{\sigma_{2}^{2}(x,y)+\delta(x,y)}}=\frac{\rho\sigma_{1}(x,y)\sigma_{2}(x,y)p}{\sqrt{\sigma_{2}^{2}(x,y)+\delta(x,y)}},
$$
then we obtain
\begin{equation}\label{hHV}
H^{0}(x,p)=\max_{y\in\mathbb{R}}V^{x,p}(y),\,\,\text{i.e.},\,\,H^{0}(x,\partial_{x}U^{0}(t,x))=\max_{y\in\mathbb{R}}V^{x,\partial_{x}U^{0}(t,x)}(y).
\end{equation}
\end{description}

We have identified the limit Hamiltonian $H^{0}$ in the three different regimes: the supercritical case (when $\alpha>2$), the critical case (when $\alpha=2$), and the subcritical case (when $\alpha<2$). According to \eqref{superU0}, \eqref{criU0} and \eqref{subU0}, it is clear that $U^{0}(t,x)$ fulfils $\partial_{t}U^{0}(t,x)=H^{0}(x,\partial_{x}U^{0}(t,x))$. By the three different expansions \eqref{sion}, \eqref{sion1} and \eqref{sion2}, we gain $U^{0}(0,x)=h(x)$. To summarize, $U^{0}(\cdot,x)$ satisfies
\begin{eqnarray}\label{sU0}
\left\{\begin{array}{l}
\partial_{t}U(t,x)=H^{0}(x,\partial_{x}U(t,x)),\quad(t,x)\in(0,T]\times\mathbb{R};\\
~~U(0,x)=h(x),\quad x\in\mathbb{R},
 \end{array}\right.
\end{eqnarray}
where the limit Hamiltonian $H^{0}$ is given by \eqref{super0U0}, \eqref{hh0} and \eqref{hHV}.

\section{\bf Hamiltonian dynamics}\label{secton4}
In this section, we deduce the main result of the paper by the comparison principle, namely, the convergence result for solutions of the Cauchy problem \eqref{sU} with the Hamiltonian $H^{\varepsilon}$ identified in Section \ref{secton2} on the compact subsets of $[0,T]\times\mathbb{R}\times\mathbb{R}$ to the unique viscosity solution of the Cauchy problem \eqref{sU0} with the limit Hamiltonian $H^{0}$ calculated in Section \ref{secton3}.

\subsection{Convergence of partial integro-differential equation}
Consider a class of compact sets $\mathcal{K}=\{\,K\times\, \Gamma:\,\text{compact}\,K,\,\Gamma\subset\subset\mathbb{R}\}$ in $\mathbb{R}\times\mathbb{R}$. Let $\{H^{\varepsilon}\}_{\varepsilon>0}$ denote a sequence of partial integro-differential operators defined on the domain $D_{+}\bigcup D_{-}$ of functions, where
\begin{equation*}
D_{\pm}:=\{{\pm}f: f\in C^{2}(\mathbb{R}^{2}),\,\lim_{r\rightarrow\infty}\inf_{|z|>r}f(z)=+\infty\}.
\end{equation*}
We will separately consider these two domains $D_{\pm}$ depending on the situation of subsolution
or supersolution. Define domains $\mathcal{D}_{\pm}$ similarly by replacing $\mathbb{R}^{2}$ with $\mathbb{R}$.
Let $\{U^{\varepsilon}\}_{\varepsilon>0}$ be the viscosity solutions of the partial integro-differential equation $\partial_{t}U=H^{\varepsilon}U$
with initial value $h\in C_{b}(\mathbb{R})$.

For above $\{U^{\varepsilon}\}_{\varepsilon>0}$, the relaxed upper semilimit is
\begin{equation*}
  U_{\uparrow}:=\sup_{y}\{\limsup_{\varepsilon\rightarrow0^{+}}U^{\varepsilon}(t_{\varepsilon},x_{\varepsilon},y):\exists\,(t_{\varepsilon},x_{\varepsilon},y)\in[0,T]\times K\times \Gamma,(t_{\varepsilon},x_{\varepsilon})\rightarrow(t,x),K\times \Gamma\in\mathcal{K}\}.
\end{equation*}
The relaxed lower semilimit $U_{\downarrow}$ can be defined analogously by replacing $\limsup$ with $\liminf$ and $\sup$ with $\inf$.
\begin{definition}\label{defi}
Let $\hat{U}$ be the upper semicontinuous regularization of $U_{\uparrow}$, and $\check{U}$ be the lower semicontinuous regularization of $U_{\downarrow}$. That is,
$$\hat{U}(t,x)=\lim_{\varepsilon\rightarrow0}\sup_{(t_{\ast},x_{\ast})\in B_{\varepsilon}(t,x)}U_{\uparrow}(t_{\ast},x_{\ast}),\quad \check{U}(t,x)=\lim_{\varepsilon\rightarrow0}\inf_{(t_{\ast},x_{\ast})\in B_{\varepsilon}(t,x)}U_{\downarrow}(t_{\ast},x_{\ast}),$$
where $B_{\varepsilon}(t,x)$ is the open ball of radius $\varepsilon$ centered at $(t,x)$.
\end{definition}
\begin{remark}
Since $h$ is bounded, the viscosity solutions $\{U^{\varepsilon}\}_{\varepsilon>0}$ are equibounded. Therefore $\hat{U}$ is bounded upper semicontinuous, and $\check{U}$ is bounded lower semicontinuous.
\end{remark}

Let $\Lambda$ be some indexing set, and
\begin{equation*}
\hat{H}^{\lambda}(x,p): \mathbb{R}\times\mathbb{R}\rightarrow\mathbb{R},\quad\check{H}^{\lambda}(x,p): \mathbb{R}\times\mathbb{R}\rightarrow\mathbb{R},\quad\text{for}\,\,\lambda\in\Lambda.
\end{equation*}
Define the limiting operators $\hat{H}$ and $\check{H}$ on domains $\mathcal{D}_{+}$ and $\mathcal{D}_{-}$ respectively, as follows\,:
\begin{equation*}
\hat{H}f(x):=\hat{H}(x,\partial_{x}f(x)),\,\,\text{for}\,f\in\mathcal{D}_{+},\quad\text{and}\quad\check{H}f(x):=\check{H}(x,\partial_{x}f(x)), \,\,\text{for}\,f\in\mathcal{D}_{-},
\end{equation*}
where
\begin{equation*}
\hat{H}(x,p):=\inf_{\lambda\in\Lambda}\hat{H}^{\lambda}(x,p)\quad\text{and}\quad\check{H}(x,p):=\sup_{\lambda\in\Lambda}\check{H}^{\lambda}(x,p).
\end{equation*}
\begin{definition}\label{vis}
(Viscosity Subsolution and Supersolution). A bounded upper semicontinuous function $\hat{U}$ is said to be a viscosity subsolution of
\begin{equation}\label{subso}
\partial_{t}U(t,x)\leq \hat{H}(x, \partial_{x}U(t,x)),
\end{equation}
if for each
\begin{equation*}
\hat{\phi}(t,x)=\phi(t)+\hat{f}(x),\quad\text{where}\,\,\phi\in C^{1}(\mathbb{R}_{+}),\, \hat{f}\in\mathcal{D}_{+},
\end{equation*}
the function $\hat{U}-\hat{\phi}$ has the global maximum point $\hat{x}$, then
\begin{equation*}
\partial_{t}\hat{\phi}(t,x)-\hat{H}(x,\partial_{x}\hat{\phi}(t,x))\leq0.
\end{equation*}
Similarly, a bounded lower semicontinuous function $\check{U}$ is said to be a viscosity supersolution of
\begin{equation}\label{superso}
\partial_{t}U(t,x)\geq\check{H}(x, \partial_{x}U(t,x)),
\end{equation}
if for each
\begin{equation*}
\check{\phi}(t,x)=\phi(t)+\check{f}(x),\quad\text{where}\,\,\phi\in C^{1}(\mathbb{R}_{+}),\, \check{f}\in\mathcal{D}_{-},
\end{equation*}
the function $\check{U}-\check{\phi}$ has the global minimum point $\check{x}$, then
\begin{equation*}
\partial_{t}\check{\phi}(t,x)-\check{H}(x,\partial_{x}\check{\phi}(t,x))\geq0.
\end{equation*}
\begin{remark}
A viscosity solution is both viscosity subsolution and supersolution when the two solutions $\hat{U}$ and $\check{U}$
are equal. Now the inequalities \eqref{subso} and \eqref{superso} in Definition \ref{vis} should be equality since the viscosity solution is defined for an equation, not an inequality.
\end{remark}
\end{definition}
The following condition will be used.
\begin{condition}\label{cond1}
(Limsup and Liminf Convergence of Operators). For each $\lambda\in\Lambda$, $\hat{f}\in\mathcal{D}_{+}$ and $\check{f}\in\mathcal{D}_{-}$ , there exist $\hat{f}_{\varepsilon}\in D_{+}$ and $\check{f}_{\varepsilon}\in D_{-}$ (which may depend on $\lambda$) such that
\begin{itemize}
        \item[(1)] for each $c>0$, there exists $ K\times \Gamma\in\mathcal{K}$ satisfying
        \begin{align}\label{con11}
           &\{(x,y): H^{\varepsilon}\hat{f}_{\varepsilon}(x,y)\geq-c\}\cap\{(x,y): \hat{f}_{\varepsilon}(x,y)\leq c\}\subset K\times \Gamma\,;\\ \label{con12}
           &\{(x,y): H^{\varepsilon}\check{f}_{\varepsilon}(x,y)\leq c\}\cap\{(x,y): \check{f}_{\varepsilon}(x,y)\geq-c\}\subset K\times \Gamma.
        \end{align}
        \item[(2)] for each $K\times \Gamma\in\mathcal{K}$,
        \begin{align}\label{lisup}
          &\lim_{\varepsilon\rightarrow0}\sup_{(x,y)\in K\times \Gamma}|\hat{f}_{\varepsilon}(x,y)-\hat{f}(x)|=0\,;\\\label{liinf}
          &\lim_{\varepsilon\rightarrow0}\sup_{(x,y)\in K\times \Gamma}|\check{f}(x)-\check{f}_{\varepsilon}(x,y)|=0.
        \end{align}
        \item[(3)] whenever $(x_{\varepsilon},y_{\varepsilon})\in K\times \Gamma\in\mathcal{K}$ satisfies $x_{\varepsilon}\rightarrow x$,
                  \begin{align}\label{limsup}
                    &\limsup_{\varepsilon\rightarrow0}H^{\varepsilon}\hat{f}_{\varepsilon}(x_{\varepsilon},y_{\varepsilon})\leq \hat{H}^{\lambda}(x,\partial_{x}\hat{f}(x))\,;\\\label{liminf}
                    &\liminf_{\varepsilon\rightarrow0}H^{\varepsilon}\check{f}_{\varepsilon}(x_{\varepsilon},y_{\varepsilon})\geq \check{H}^{\lambda}(x,\partial_{x}\check{f}(x)).
                  \end{align}
        \end{itemize}
\end{condition}
In this case the following convergence results for $\{U^{\varepsilon}\}_{\varepsilon>0}$ as $\varepsilon\rightarrow0$ hold.
\begin{lemma}\label{subsup}
Suppose that $\sup\limits_{\varepsilon>0}||U^{\varepsilon}||_{\infty}<\infty$, i.e., the viscosity solutions $\{U^{\varepsilon}\}_{\varepsilon>0}$ of the partial integro-differential equation
$$
\partial_{t}U(t,x,y)=H^{\varepsilon}U(t,x,y),\quad U(0,x,y)=h(x)\in C_{b}(\mathbb{R})
$$
are uniformly bounded. Then, under Condition \ref{cond1}, $\hat{U}$ is a subsolution of
\eqref{subso} and $\check{U}$ is a supersolution of \eqref{superso} with the same initial values, where $\hat{U}$ and $\check{U}$ are given by Definition \ref{defi}.
\end{lemma}
\begin{proof}
Let $\hat{\phi}(t,x)=\phi(t)+\hat{f}(x)$ with $\phi\in C^{1}(\mathbb{R}_{+})$ and $\hat{f}\in\mathcal{D}_{+}$ fixed, and $\lambda\in\Lambda$ be given. Denote the global maximum of $\hat{U}-\hat{\phi}$ by $(\hat{t},\hat{x})$ with $\hat{t}>0$. Take $\hat{\phi}^{\varepsilon}(t,x,y)=\phi(t)+\hat{f}_{\varepsilon}(x,y)$,
where $\hat{f}_{\varepsilon}$ is the approximate of $\hat{f}$ in \eqref{lisup}, then $\hat{\phi}^{\varepsilon}$ has compact level sets. In addition,
combining with the boundness of $U^{\varepsilon}$, there exists $(t_{\varepsilon},x_{\varepsilon},y_{\varepsilon})\in[0,T]\times\mathbb{R}\times\mathbb{R}$ such that
\begin{equation}\label{in}
(U^{\varepsilon}-\hat{\phi}^{\varepsilon})(t_{\varepsilon},x_{\varepsilon},y_{\varepsilon})\geq(U^{\varepsilon}-\hat{\phi}^{\varepsilon})(t,x,y), \quad\text{for}~~(t,x,y)\in[0,T]\times\mathbb{R}\times\mathbb{R}
\end{equation}
and
\begin{equation}\label{ine}
\partial_{t}\phi(t_{\varepsilon})-H^{\varepsilon}\hat{f}_{\varepsilon}(x_{\varepsilon},y_{\varepsilon})\leq0,
\end{equation}
which implies
\begin{equation}\label{infe}
\inf_{\varepsilon>0}H^{\varepsilon}\hat{f}_{\varepsilon}(x_{\varepsilon},y_{\varepsilon})>-\infty.
\end{equation}
Take $(t_{1},x_{1},y_{1})\in[0,T]\times\mathbb{R}\times\mathbb{R}$ such that $\hat{\phi}(t_{1},x_{1})<\infty$, then
\begin{equation}\label{hi}
\hat{\phi}^{\varepsilon}(t_{1},x_{1},y_{1})=\phi(t_{1})+\hat{f}_{\varepsilon}(x_{1},y_{1})\rightarrow\phi(t_{1})+\hat{f}(x_{1})=\hat{\phi}(t_{1},x_{1})<\infty.
\end{equation}
Using \eqref{in} and \eqref{hi}, we get
\begin{equation*}
\hat{\phi}^{\varepsilon}(t_{\varepsilon},x_{\varepsilon},y_{\varepsilon})\leq2\sup_{\varepsilon>0}||U^{\varepsilon}||_{\infty}+\sup_{\varepsilon>0}\hat{\phi}^{\varepsilon}(t_{1},x_{1},y_{1})<\infty,
\end{equation*}
therefore
\begin{equation}\label{su}
\sup_{\varepsilon>0}\hat{f}_{\varepsilon}(x_{\varepsilon},y_{\varepsilon})<\infty.
\end{equation}
By \eqref{infe} and \eqref{su}, for $c>0$, we have $H^{\varepsilon}\hat{f}_{\varepsilon}(x_{\varepsilon},y_{\varepsilon})\geq-c$ and $\hat{f}_{\varepsilon}(x_{\varepsilon},y_{\varepsilon})\leq c$. Based on \eqref{con11} in Condition \ref{cond1}, there exists $K\times \Gamma\in\mathcal{K}$ such that $(x_{\varepsilon},y_{\varepsilon})\in K\times \Gamma$.

Since $K\times \Gamma$ is compact, there exists a subsequence of $\{(t_{\varepsilon},x_{\varepsilon},y_{\varepsilon})\}_{\varepsilon>0}$ (for simplify, we still use $\{(t_{\varepsilon},x_{\varepsilon},y_{\varepsilon})\}_{\varepsilon>0}$ to index it) and a $(t_{*}, x_{*})\in[0,T]\times\mathbb{R}$ such that $t_{\varepsilon}\rightarrow t_{*}$ and $x_{\varepsilon}\rightarrow x_{*}$. From the definition of $\hat{U}$, \eqref{lisup} and \eqref{in}, we have
\begin{equation*}
(\hat{U}-\hat{\varphi})(t_{*},x_{*})\geq(\hat{U}-\hat{\varphi})(t,x),
\end{equation*}
which indicates that $(t_{*}, x_{*})$ has to be the unique global maximizer $(\hat{t}, \hat{x})$ for $\hat{U}-\hat{\varphi}$ that appeared earlier.
In other words, there exists a subsequence of $\{(t_{\varepsilon},x_{\varepsilon},y_{\varepsilon})\}_{\varepsilon>0}$ such that $t_{\varepsilon}\rightarrow \hat{t}$ and $x_{\varepsilon}\rightarrow \hat{x}$. From \eqref{limsup} and \eqref{ine}, we obtain
\begin{equation*}
\partial_{t}\hat{\varphi}(\hat{t})\leq \hat{H}^{\lambda}(\hat{x},\partial_{x}\hat{f}(\hat{x})).
\end{equation*}
Take $\inf_{\lambda\in\Lambda}$ on both sides, we get
\begin{equation*}
\partial_{t}\hat{\phi}(\hat{t},\hat{x})-\inf_{\lambda\in\Lambda}\hat{H}^{\lambda}(\hat{x},\partial_{x}\hat{\phi}(\hat{t},\hat{x}))\leq0,
\end{equation*}
which shows that $\hat{U}$ is a subsolution of \eqref{subso}. Similarly, we can proof that $\check{U}$ is a supersolution of \eqref{superso} under Condition \ref{cond1}.
\end{proof}

\begin{lemma}\label{son}
Let $\hat{U}$ and $\check{U}$ be defined as in Definition \ref{defi}. If a comparison principle between subsolution of \eqref{subso} and supersolution of \eqref{superso} holds, that is, if every subsolution of \eqref{subso} is less than or equal to every supersolution of \eqref{superso}, then $\hat{U}=\check{U}$ and $U^{\varepsilon}(t,x,y)\rightarrow U^{0}(t,x)$ uniformly over compact subsets of $[0,T]\times\mathbb{R}\times\mathbb{R}$ as $\varepsilon\rightarrow0$, where $U^{0}:=\hat{U}=\check{U}$.
\end{lemma}

\begin{proof}
The comparison principle gives $\hat{U}\leq\check{U}$, while by construction we have $\check{U}\leq\hat{U}$. Then we obtain uniform convergence of $U^{\varepsilon}\rightarrow U^{0}:=\hat{U}=\check{U}$ over compact subsets of $[0,T]\times\mathbb{R}\times\mathbb{R}$.
\end{proof}

\subsection{Convergence of $H^{\varepsilon}$}\label{CH}

To verify that Condition \ref{cond1} holds for $H^{\varepsilon}$ defined by \eqref{Hed}, we need to identify the correct indexing set $\Lambda$,
the family of operators $\hat{H}^{\lambda}$ and $\check{H}^{\lambda}$, and the appropriate test functions $\hat{f}_{\varepsilon}$ and $\check{f}_{\varepsilon}$ for each given $\hat{f}$ and $\check{f}$, respectively.

Now we let
\begin{equation*}
\Lambda:=\{\lambda=(\zeta,\theta): \zeta\in C_{c}^{2}(\mathbb{R}),\,0<\theta<1\},
\end{equation*}
and define two domains
\begin{equation*}
\hat{D}_{\pm}:=\{f\in C^{2}(\mathbb{R}): f(x)=\varphi(x)\pm\beta\ln(1+x^{2}), \varphi\in C_{c}^{2}(\mathbb{R}),\beta>0\}.
\end{equation*}

\textbf{The supercritical case $(\alpha>2)$:} For each $\hat{f}\in\hat{D}_{+}$ and $\lambda=(\zeta,\theta)\in\Lambda$, we let $g(y):=\zeta(y)+\theta\xi(y)$, where $\xi$ is the Lyapunov function satisfying $(\textbf{C}_{5})-(\textbf{i})$, and define a sequence of test functions
\begin{equation}\label{ff}
\hat{f}_{\varepsilon}(x,y):=\hat{f}(x)+\varepsilon^{\alpha-1}g(y)=\hat{f}(x)+\varepsilon^{\alpha-1}\zeta(y)+\varepsilon^{\alpha-1}\theta\xi(y).
\end{equation}
Then we obtain
{\small\begin{align} \nonumber
H^{\varepsilon}\hat{f}_{\varepsilon}(x,y)
&=\varepsilon\left(b_{1}(x,y)\partial_{x}\hat{f}(x)+\sigma_{1}^{2}(x,y)\partial_{xx}^{2}\hat{f}(x)\right)+\sigma_{1}^{2}(x,y)\big(\partial_{x}\hat{f}(x)\big)^{2}\\ \nonumber
&~~~+\int_{\mathbb{R}\setminus \{0\}}\Big(e^{\frac{1}{\varepsilon}[\hat{f}(x+\varepsilon k_{1}(x,y,z))-\hat{f}(x)]}-1-k_{1}(x,y,z)\partial_{x}\hat{f}(x)\Big)\nu_{1}(dz)\\ \nonumber
&~~~+2\varepsilon^{\frac{\alpha}{2}-1}\rho\sigma_{1}(x,y)\sigma_{2}(x,y)\partial_{x}\hat{f}(x)\partial_{y}g(y)+\varepsilon^{2-\alpha}e^{-\varepsilon^{\alpha-2}g(y)}\mathcal{L}^{x}e^{\varepsilon^{\alpha-2}g(y)}\\[1mm]  \nonumber
&=\varepsilon\left(b_{1}(x,y)\partial_{x}\hat{f}(x)+\sigma_{1}^{2}(x,y)\partial_{xx}^{2}\hat{f}(x)\right)+\sigma_{1}^{2}(x,y)\big(\partial_{x}\hat{f}(x)\big)^{2}+b_{2}(x,y)\partial_{y}g(y)\\[1mm] \nonumber
&~~~+2\varepsilon^{\frac{\alpha}{2}-1}\rho\sigma_{1}(x,y)\sigma_{2}(x,y)\partial_{x}\hat{f}(x)\partial_{y}g(y)+\varepsilon^{\alpha-2}\sigma_{2}^{2}(x,y)\big(\partial_{y}g(y)\big)^{2}+\sigma_{2}^{2}(x,y)\partial_{yy}^{2}g(y)\\ \nonumber
&~~~+\int_{\mathbb{R}\setminus \{0\}}\Big(e^{\frac{1}{\varepsilon}[\hat{f}(x+\varepsilon k_{1}(x,y,z))-\hat{f}(x)]}-1-k_{1}(x,y,z)\partial_{x}\hat{f}(x)\Big)\nu_{1}(dz)\\ \label{alsup}
&~~~+\varepsilon^{2-\alpha}\int_{\mathbb{R}\setminus \{0\}}\Big(e^{\varepsilon^{\alpha-2}[g(y+k_{2}(x,y,z))-g(y)]}-1-\varepsilon^{\alpha-2}k_{2}(x,y,z)\partial_{y}g(y)\Big)\nu_{2}(dz),
\end{align}}
where $\mathcal{L}^{x}$ defined by \eqref{LLx}. Note that $||\partial_{x}\hat{f}||_{\infty}+||\partial^{2}_{xx}\hat{f}||_{\infty}<\infty$ and $||\partial_{y}g||_{\infty}+||\partial^{2}_{yy}g||_{\infty}<\infty$, by $(\textbf{C}_{5})-(\textbf{i})$ and \eqref{alsup}, there exist constants $C, \tilde{C}>0$ such that
\begin{equation*}
H^{\varepsilon}\hat{f}_{\varepsilon}(x,y)\leq V^{x,\partial_{x}\hat{f}(x)}(y)+b_{2}(x,y)\partial_{y}\zeta(y)+\sigma_{2}^{2}(x,y)\partial_{yy}^{2}\zeta(y)-C\theta\xi(y)+\tilde{C}\varepsilon.
\end{equation*}
We also have $\hat{f}_{\varepsilon}(x,y)\geq \hat{f}(x)-\varepsilon^{\alpha-1}||\zeta||_{\infty}$, then for each $c>0$, we can find $K\times\Gamma\in\mathcal{K}$, such that
\begin{equation*}
\{(x,y): H^{\varepsilon}\hat{f}_{\varepsilon}(x,y)\geq-c\}\cap\{(x,y): \hat{f}_{\varepsilon}(x,y)\leq c\}\subset K\times \Gamma,
\end{equation*}
which proves \eqref{con11} in Condition \ref{cond1}. By \eqref{ff} and $\alpha>1$,
$$
\hat{f}_{\varepsilon}(x,y)=\hat{f}(x)+\varepsilon^{\alpha-1}g(y)\rightarrow \hat{f}(x)\quad\text{as}~~\varepsilon\rightarrow0,
$$
we immediately obtain that
\eqref{lisup} holds. Furthermore, for $\lambda=(\zeta,\theta)\in\Lambda$, by taking
\begin{equation}\label{Hl1}
\hat{H}^{\lambda}(x,p)=\sup\limits_{y\in\mathbb{R}}\{\,V^{x,p}(y)+b_{2}(x,y)\partial_{y}\zeta(y)+\sigma_{2}^{2}(x,y)\partial_{yy}^{2}\zeta(y)-C\theta\xi(y)\,\},
\end{equation}
then for any sequence $(x_{\varepsilon},y_{\varepsilon})\in K\times \Gamma\in\mathcal{K}$ satisfying $x_{\varepsilon}\rightarrow x$, we have
\begin{equation*}
\limsup_{\varepsilon\rightarrow0}H^{\varepsilon}\hat{f}_{\varepsilon}(x_{\varepsilon},y_{\varepsilon})\leq \hat{H}^{\lambda}(x,\partial_{x}\hat{f}(x)),
\end{equation*}
which implies that \eqref{limsup} holds.

The rest of Condition \ref{cond1} can be verified similarly. Define a sequence of functions
\begin{equation*}
\check{f}_{\varepsilon}(x,y)=\check{f}(x)+\varepsilon^{\alpha-1}g(y)\quad\text{with}~~g(y):=\zeta(y)-\theta\xi(y),
\end{equation*}
for each $\check{f}\in\hat{D}_{-}$ and $\lambda=(\zeta,\theta)\in\Lambda$. Notice that \eqref{con12} and \eqref{liinf} hold by the same arguments as above. Take
\begin{equation}\label{Hm1}
\check{H}^{\lambda}(x,p)=\inf\limits_{y\in\mathbb{R}}\{\,V^{x,p}(y)+b_{2}(x,y)\partial_{y}\zeta(y)+\sigma_{2}^{2}(x,y)\partial_{yy}^{2}\zeta(y)+C\theta\xi(y)\,\},
\end{equation}
then for any sequence $(x_{\varepsilon},y_{\varepsilon})\in K\times \Gamma\in\mathcal{K}$ such that $x_{\varepsilon}\rightarrow x$, we get
\begin{equation*}
\liminf_{\varepsilon\rightarrow0}H^{\varepsilon}\check{f}_{\varepsilon}(x_{\varepsilon},y_{\varepsilon})\geq \check{H}^{\lambda}(x,\partial_{x}\check{f}(x)),
\end{equation*}
thus \eqref{liminf} holds.

\textbf{The critical case ($\alpha=2$):} For each $\hat{f}\in\hat{D}_{+}$ and $\lambda=(\zeta,\theta)\in\Lambda$, define a sequence of functions
\begin{equation}\label{f}
\hat{f}_{\varepsilon}(x,y)=\hat{f}(x)+\varepsilon g(y),
\end{equation}
where $g(y):=(1-\theta)\zeta(y)+\theta\xi(y)$, and $\xi$ is defined as before in $(\textbf{C}_{5})-(\textbf{ii})$. Then
{\small\begin{align} \nonumber
H^{\varepsilon}\hat{f}_{\varepsilon}(x,y)
&=e^{-g(y)}\mathcal{L}^{x,\partial_{x}\hat{f}(x)}e^{g(y)}+\varepsilon\left(\,b_{1}(x,y)\partial_{x}\hat{f}(x)+\sigma_{1}^{2}(x,y)\partial_{xx}^{2}\hat{f}(x)\,\right)+\sigma_{1}^{2}(x,y)\big(\partial_{x}\hat{f}(x)\big)^{2}\\ \nonumber
&~~~+\int_{\mathbb{R}\setminus \{0\}}\Big(e^{\frac{1}{\varepsilon}[\hat{f}(x+\varepsilon k_{1}(x,y,z))-\hat{f}(x)]}-1-k_{1}(x,y,z)\partial_{x}\hat{f}(x)\Big)\nu_{1}(dz)\\ \nonumber
&\leq(1-\theta)e^{-\zeta(y)}\mathcal{L}^{x,\partial_{x}\hat{f}(x)}e^{\zeta(y)}+\theta e^{-\xi(y)}\mathcal{L}^{x,\partial_{x}\hat{f}(x)}e^{\xi(y)}+ \varepsilon \left(\,b_{1}(x,y)\partial_{x}\hat{f}(x)+\sigma_{1}^{2}(x,y)\partial_{xx}^{2}\hat{f}(x)\,\right)\\ \label{alc}
&~~~+\sigma_{1}^{2}(x,y)\big(\partial_{x}\hat{f}(x)\big)^{2}+\int_{\mathbb{R}\setminus \{0\}}\big(e^{\frac{1}{\varepsilon}[\hat{f}(x+\varepsilon k_{1}(x,y,z))-\hat{f}(x)]}-1-k_{1}(x,y,z)\partial_{x}\hat{f}(x)\big)\nu_{1}(dz),
\end{align}}
where $\mathcal{L}^{x,\partial_{x}\hat{f}(x)}$ as in \eqref{LLp} with $p=\partial_{x}\hat{f}(x)$.
By the choice of domain $\hat{D}_{+}$, then $\hat{f}\in\hat{D}_{+}$ has compact level sets in $\mathbb{R}$ and $||\partial_{x}\hat{f}||_{\infty}+||\partial^{2}_{xx}\hat{f}||_{\infty}<\infty$. Based on $(\textbf{C}_{5})-(\textbf{ii})$ and $\zeta\in C_{c}^{2}(\mathbb{R})$, there exists $C>0$ such that
\begin{equation*}
H^{\varepsilon}\hat{f}_{\varepsilon}(x,y)\leq V^{x,\partial_{x}\hat{f}(x)}(y)+(1-\theta)e^{-\zeta(y)}\mathcal{L}^{x,\partial_{x}\hat{f}(x)}e^{\zeta(y)}+\theta e^{-\xi(y)}\mathcal{L}^{x,\partial_{x}\hat{f}(x)}e^{\xi(y)}+C\varepsilon.
\end{equation*}
 and $\hat{f}_{\varepsilon}(x,y)\geq f(x)-\varepsilon||\zeta||_{\infty}$. For each $c>0$, we can find $K\times\Gamma\in\mathcal{K}$, such that
\begin{equation*}
\{(x,y): H^{\varepsilon}\hat{f}_{\varepsilon}(x,y)\geq-c\}\cap\{(x,y): \hat{f}_{\varepsilon}(x,y)\leq c\}\subset K\times \Gamma,
\end{equation*}
which proves \eqref{con11} in Condition \ref{cond1}. From the definition of $\hat{f}_{\varepsilon}(x,y)$ in \eqref{f}, we immediately obtain that
\eqref{lisup} holds. Moreover, we define the family of operators $\hat{H}^{\lambda}(x,p)$ by
\begin{equation}\label{H02}
\hat{H}^{\lambda}(x,p)=\sup\limits_{y\in\mathbb{R}}\{\,V^{x,p}(y)+(1-\theta)e^{-\zeta(y)}\mathcal{L}^{x,p} e^{\zeta(y)}+\theta e^{-\xi(y)}\mathcal{L}^{x,p}e^{\xi(y)}\,\},
\end{equation}
where $\mathcal{L}^{x,p}$ and $V^{x,p}(y)$ are defined by \eqref{LLp} and \eqref{v} respectively. By \eqref{alc}, for any sequence $(x_{\varepsilon},y_{\varepsilon})\in K\times \Gamma\in\mathcal{K}$ satisfying $x_{\varepsilon}\rightarrow x$, we have
\begin{equation*}
\limsup_{\varepsilon\rightarrow0}H^{\varepsilon}\hat{f}_{\varepsilon}(x_{\varepsilon},y_{\varepsilon})\leq \hat{H}^{\lambda}(x,\partial_{x}\hat{f}(x)),
\end{equation*}
which implies that \eqref{limsup} holds.

The proof of the rest of Condition \ref{cond1} follows by straightforward modifications. Define a sequence of functions
\begin{equation*}
\check{f}_{\varepsilon}(x,y)=\check{f}(x)+\varepsilon g(y)\quad\text{with}~~g(y):=(1+\theta)\zeta(y)-\theta\xi(y),
\end{equation*}
for each $\check{f}\in\hat{D}_{-}$ and $\lambda=(\zeta,\theta)\in\Lambda$. Please remark that \eqref{con12} and \eqref{liinf} hold by the same arguments as above. Take
\begin{equation}\label{H22}
\check{H}^{\lambda}(x,p)=\inf\limits_{y\in\mathbb{R}}\{\,V^{x,p}(y)+(1+\theta)e^{-\zeta(y)}\mathcal{L}^{x,p} e^{\zeta(y)}-\theta e^{-\xi(y)}\mathcal{L}^{x,p}e^{\xi(y)}\,\},
\end{equation}
then for any sequence $(x_{\varepsilon},y_{\varepsilon})\in K\times \Gamma\in\mathcal{K}$ such that $x_{\varepsilon}\rightarrow x$, we obtain
\begin{equation*}
\liminf_{\varepsilon\rightarrow0}H^{\varepsilon}\check{f}_{\varepsilon}(x_{\varepsilon},y_{\varepsilon})\geq \check{H}^{\lambda}(x,\partial_{x}\check{f}(x)),
\end{equation*}
Hence \eqref{liminf} holds.

\textbf{The subcritical case ($\alpha<2$):} For each $\hat{f}\in\hat{D}_{+}$ and $\lambda=(\zeta,\theta)\in\Lambda$, define a sequence of functions
\begin{equation}\label{fff}
\hat{f}_{\varepsilon}(x,y)=\hat{f}(x)+\varepsilon^{\frac{\alpha}{2}}g(y),
\end{equation}
where $g(y):=(1-\theta)\zeta(y)+\theta\xi(y)$, and $\xi$ is the Lyapunov function on $\mathbb{R}$ satisfying $(\textbf{C}_{5})-(\textbf{iii})$.
Then
{\small\begin{align} \nonumber
& H^{\varepsilon}\hat{f}_{\varepsilon}(x,y)\\
&=\varepsilon^{2-\alpha}e^{-\varepsilon^{\frac{\alpha}{2}-1}g(y)}\mathcal{L}^{x}e^{\varepsilon^{\frac{\alpha}{2}-1}g(y)}+\varepsilon\left(b_{1}(x,y)\partial_{x}\hat{f}(x)+\sigma_{1}^{2}(x,y)\partial_{xx}^{2}\hat{f}(x)\right)+\sigma_{1}^{2}(x,y)\big(\partial_{x}\hat{f}(x)\big)^{2}\\ \nonumber
&~~~+2\rho\sigma_{1}(x,y)\sigma_{2}(x,y)\partial_{x}\hat{f}(x)\partial_{y}g(y)+\int_{\mathbb{R}\setminus \{0\}}\Big(e^{\frac{1}{\varepsilon}[\hat{f}(x+\varepsilon k_{1}(x,y,z))-\hat{f}(x)]}-1-k_{1}(x,y,z)\partial_{x}\hat{f}(x)\Big)\nu_{1}(dz)\\ \nonumber
&=\varepsilon\left(b_{1}(x,y)\partial_{x}\hat{f}(x)+\sigma_{1}^{2}(x,y)\partial_{xx}^{2}\hat{f}(x)\right)+\sigma_{1}^{2}(x,y)\big(\partial_{x}\hat{f}(x)\big)^{2}+\varepsilon^{1-\frac{\alpha}{2}}b_{2}(x,y)\partial_{y}g(y)\\[1mm] \nonumber
&~~~+2\rho\sigma_{1}(x,y)\sigma_{2}(x,y)\partial_{x}\hat{f}(x)\partial_{y}g(y)+\sigma_{2}^{2}(x,y)\big(\partial_{y}g(y)\big)^{2}+\varepsilon^{1-\frac{\alpha}{2}}\sigma_{2}^{2}(x,y)\partial_{yy}^{2}g(y)\\ \nonumber
&~~~+\int_{\mathbb{R}\setminus \{0\}}\Big(e^{\frac{1}{\varepsilon}[\hat{f}(x+\varepsilon k_{1}(x,y,z))-\hat{f}(x)]}-1-k_{1}(x,y,z)\partial_{x}\hat{f}(x)\Big)\nu_{1}(dz)\\ \label{alsub}
&~~~+\varepsilon^{2-\alpha}\int_{\mathbb{R}\setminus \{0\}}\Big(e^{\varepsilon^{\frac{\alpha}{2}-1}[g(y+k_{2}(x,y,z))-g(y)]}-1-\varepsilon^{\frac{\alpha}{2}-1}k_{2}(x,y,z)\partial_{y}g(y)\Big)\nu_{2}(dz),
\end{align}}
where $\mathcal{L}^{x}$ as in \eqref{LLx}. Note that $||\partial_{x}\hat{f}||_{\infty}+||\partial^{2}_{xx}\hat{f}||_{\infty}<\infty$ and $||\partial_{y}g||_{\infty}+||\partial^{2}_{yy}g||_{\infty}<\infty$, by $(\textbf{C}_{5})-(\textbf{iii})$ and \eqref{alsub}, there exists constant $C>0$ such that
\begin{align*} \nonumber
H^{\varepsilon}\hat{f}_{\varepsilon}(x,y)
&\leq V^{x,\partial_{x}\hat{f}(x)}(y)+(1-\theta)[\,p(x,y)\partial_{x}\hat{f}(x)\partial_{y}\zeta(y)+q(x,y)\big(\partial_{y}\zeta(y)\big)^{2}\,]
\\ \nonumber
&~~~~+\theta[\,p(x,y)\partial_{x}\hat{f}(x)\partial_{y}\xi(y)+q(x,y)\big(\partial_{y}\xi(y)\big)^{2}\,]+C\varepsilon,
\end{align*}
where
\begin{equation*}
p(x,y)=2\rho\sigma_{1}(x,y)\sigma_{2}(x,y)\quad\text{and}\quad q(x,y)=\sigma_{2}^{2}(x,y)+\frac{1}{2}\int_{\mathbb{R}\setminus\{0\}}k_{2}^{2}(x,y,z)\nu_{2}(dz).
\end{equation*}
We also have $\hat{f}_{\varepsilon}(x,y)\geq \hat{f}(x)-\varepsilon^{\frac{\alpha}{2}}||\zeta||_{\infty}$, then for each $c>0$, we can find $K\times\Gamma\in\mathcal{K}$, such that
\begin{equation*}
\{(x,y): H^{\varepsilon}\hat{f}_{\varepsilon}(x,y)\geq-c\}\cap\{(x,y): \hat{f}_{\varepsilon}(x,y)\leq c\}\subset K\times \Gamma,
\end{equation*}
which proves \eqref{con11} in Condition \ref{cond1}. By the definition of $\hat{f}_{\varepsilon}(x,y)$ in \eqref{fff}, we immediately obtain that
\eqref{lisup} holds. Furthermore, for $\lambda=(\zeta,\theta)\in\Lambda$, by taking
\begin{align} \nonumber
\hat{H}^{\lambda}(x,p)
&=\sup\limits_{y\in\mathbb{R}}\{\,V^{x,p}(y)+(1-\theta)[\,p(x,y)\partial_{y}\zeta(y)p+q(x,y)\big(\partial_{y}\zeta(y)\big)^{2}\,]
\\ \label{Hl3}
&~~~~+\theta[\,p(x,y)\partial_{y}\xi(y)p+q(x,y)\big(\partial_{y}\xi(y)\big)^{2}\,]\,\},
\end{align}
then for any sequence $(x_{\varepsilon},y_{\varepsilon})\in K\times \Gamma\in\mathcal{K}$ satisfying $x_{\varepsilon}\rightarrow x$, we have
\begin{equation*}
\limsup_{\varepsilon\rightarrow0}H^{\varepsilon}\hat{f}_{\varepsilon}(x_{\varepsilon},y_{\varepsilon})\leq \hat{H}^{\lambda}(x,\partial_{x}\hat{f}(x)),
\end{equation*}
which implies that \eqref{limsup} holds.

The rest of Condition \ref{cond1} can be treated in a similar manner. Define a sequence of functions
\begin{equation*}
\check{f}_{\varepsilon}(x,y)=\check{f}(x)+\varepsilon^{\frac{\alpha}{2}}g(y)\quad\text{with}~~g(y):=(1+\theta)\zeta(y)-\theta\xi(y),
\end{equation*}
for each $\check{f}\in\hat{D}_{-}$ and $\lambda=(\zeta,\theta)\in\Lambda$. Please note that \eqref{con12} and \eqref{liinf} hold by the same arguments as above. Take
\begin{align} \nonumber
\check{H}^{\lambda}(x,p)
&=\inf\limits_{y\in\mathbb{R}}\{\,V^{x,p}(y)+(1+\theta)[\,p(x,y)\partial_{y}\zeta(y)p+q(x,y)\big(\partial_{y}\zeta(y)\big)^{2}\,]
\\ \label{Hm3}
&~~~~-\theta[\,p(x,y)\partial_{y}\xi(y)p+q(x,y)\big(\partial_{y}\xi(y)\big)^{2}\,]\,\},
\end{align}
then for any sequence $(x_{\varepsilon},y_{\varepsilon})\in K\times \Gamma\in\mathcal{K}$ such that $x_{\varepsilon}\rightarrow x$, we get
\begin{equation*}
\liminf_{\varepsilon\rightarrow0}H^{\varepsilon}\check{f}_{\varepsilon}(x_{\varepsilon},y_{\varepsilon})\geq \check{H}^{\lambda}(x,\partial_{x}\check{f}(x)),
\end{equation*}
thus \eqref{liminf} holds.

\subsection{Comparison theorem for $H^{0}$}\label{CT}
The comparison theorem among viscosity subsolution and supersolution of $\partial_{t}U(t,x)=H^{0}(x,\partial_{x}U(t,x))$ in \eqref{sU0} will be the crucial tool for proving that the convergence of $\{U^{\varepsilon}\}_{\varepsilon>0}$ described by \eqref{sU} is not only in the weak sense of semilimits but in fact uniform, and the limit is unique.

We need to derive that the comparison principle holds for $H^{0}$, which is one of the key conditions for Lemma \ref{con}.
And employ it afterwards, equation \eqref{sU0} has a unique viscosity solution  for given initial values $U(0,\cdot)$ and $T>0$.
\begin{theorem}\label{unf}
Let $\hat{U}$ and $\check{U}$ be, respectively, a bounded upper semicontinuous viscosity subsolution and a bounded lower semicontinuous viscosity supersolution to $\partial_{t}U(t,x)=H^{0}(x,\partial_{x}U(t,x))$ such that $\hat{U}(0,x)\leq \check{U}(0,x)$ for all $x\in\mathbb{R}$, and $H^{0}$ is uniformly continuous on compact sets. Then $\hat{U}(t,x)\leq \check{U}(t,x)$ for all $(t,x)\in[0,T]\times\mathbb{R}$.
\end{theorem}
\begin{proof}
By contradiction, we assume that there exists $x\in\mathbb{R}$ such that
\begin{equation}\label{contra}
\sup_{t\in[0,T]}\{\hat{U}(t,x)-\check{U}(t,x)\}>C>0.
\end{equation}
For $\gamma>0$ , define
{\small\begin{equation}\label{Phi}
Q(t,s,x,y)=\hat{U}(t,x)-\check{U}(s,y)-\frac{|t-s|^{2}+|x-y|^{2}}{2\varepsilon}+\gamma\left(\ln(1+|x|^{2})+\ln(1+|y|^{2})\right)-Cs.
\end{equation}}
For $\varepsilon>0$ small enough, $Q$ has a maximum point that we denote with $(t'_{\varepsilon},s'_{\varepsilon},x'_{\varepsilon},y'_{\varepsilon})$. Since $\hat{U}$, $\check{U}$ are bounded, there exists $R_{\gamma}>0$ such that $|x'_{\varepsilon}|, |y'_{\varepsilon}|\leq R_{\gamma}$.

If either $t'_{\varepsilon}=0$ or $s'_{\varepsilon}=0$, we get a contradiction with \eqref{contra} by means of the inequality $\hat{U}(0,x)\leq \check{U}(0,x)$ for all $x\in\mathbb{R}$. So we consider the case $(t'_{\varepsilon},s'_{\varepsilon},x'_{\varepsilon},y'_{\varepsilon})\in(0,T]\times(0,T]\times\mathbb{R}\times\mathbb{R}$.
Taking $(t,s,x,y)=(t'_{\varepsilon},t'_{\varepsilon},x'_{\varepsilon},x'_{\varepsilon})$ in \eqref{Phi}, we get
\begin{equation*}
Q(t'_{\varepsilon},t'_{\varepsilon},x'_{\varepsilon},x'_{\varepsilon})=\hat{U}(t'_{\varepsilon},x'_{\varepsilon})-\check{U}(t'_{\varepsilon},x'_{\varepsilon})+2\gamma\ln(1+|x'_{\varepsilon}|^{2}) -Ct'_{\varepsilon}.
\end{equation*}
Similarly, we have
\begin{equation*}
Q(s'_{\varepsilon},s'_{\varepsilon},y'_{\varepsilon},y'_{\varepsilon})=\hat{U}(s'_{\varepsilon},y'_{\varepsilon})-\check{U}(s'_{\varepsilon},y'_{\varepsilon})+2\gamma\ln(1+|y'_{\varepsilon}|^{2}) -Cs'_{\varepsilon}.
\end{equation*}
Utilizing
\begin{equation*}
Q(t'_{\varepsilon},t'_{\varepsilon},x'_{\varepsilon},x'_{\varepsilon})+Q(s'_{\varepsilon},s'_{\varepsilon},y'_{\varepsilon},y'_{\varepsilon})\leq2Q(t'_{\varepsilon},s'_{\varepsilon},x'_{\varepsilon},y'_{\varepsilon}),
\end{equation*}
we obtain
{\small\begin{align*}
&\hat{U}(t'_{\varepsilon},x'_{\varepsilon})+\hat{U}(s'_{\varepsilon},y'_{\varepsilon})-\check{U}(t'_{\varepsilon},x'_{\varepsilon})-\check{U}(s'_{\varepsilon},y'_{\varepsilon})+2\gamma\left(\ln(1+|x'_{\varepsilon}|^{2})+\ln(1+|y'_{\varepsilon}|^{2})\right)-C(t'_{\varepsilon}+s'_{\varepsilon})\\
&\leq2\hat{U}(t'_{\varepsilon},x'_{\varepsilon})-2\check{U}(s'_{\varepsilon},y'_{\varepsilon})-\frac{|t'_{\varepsilon}-s'_{\varepsilon}|^{2}+|x'_{\varepsilon}-y'_{\varepsilon}|^{2}}{\varepsilon}+2\gamma\left(\ln(1+|x'_{\varepsilon}|^{2})+\ln(1+|y'_{\varepsilon}|^{2})\right)-2Cs'_{\varepsilon}.
\end{align*}}
Then
\begin{align*}
\frac{|t'_{\varepsilon}-s'_{\varepsilon}|^{2}+|x'_{\varepsilon}-y'_{\varepsilon}|^{2}}{\varepsilon}
&\leq \hat{U}(t'_{\varepsilon},x'_{\varepsilon})-\hat{U}(s'_{\varepsilon},y'_{\varepsilon})+\check{U}(t'_{\varepsilon},x'_{\varepsilon})-\check{U}(s'_{\varepsilon},y'_{\varepsilon})+C(t'_{\varepsilon}-s'_{\varepsilon})\\
&\leq 2||\hat{U}||_{\infty}+2||\check{U}||_{\infty}+2CT:=C_{1}<\infty,
\end{align*}
which implies
\begin{equation*}
|t'_{\varepsilon}-s'_{\varepsilon}|^{2}+|x'_{\varepsilon}-y'_{\varepsilon}|^{2}\leq C_{1} \varepsilon .
\end{equation*}
Hence $|t'_{\varepsilon}-s'_{\varepsilon}|, |x'_{\varepsilon}-y'_{\varepsilon}|\rightarrow0$ as $\varepsilon\rightarrow0$.

Let
\begin{equation*}
Q_{1}(t,x):=\check{U}(s'_{\varepsilon},y'_{\varepsilon})+\frac{|t-s'_{\varepsilon}|^{2}+|x-y'_{\varepsilon}|^{2}}{2\varepsilon}-\gamma\left(\ln(1+|x|^{2})+\ln(1+|y'_{\varepsilon}|^{2})\right)+Cs'_{\varepsilon}
\end{equation*}
and
\begin{equation*}
Q_{2}(s,y):=\hat{U}(t'_{\varepsilon},x'_{\varepsilon})-\frac{|t'_{\varepsilon}-s|^{2}+|x'_{\varepsilon}-y|^{2}}{2\varepsilon}+\gamma\left(\ln(1+|x'_{\varepsilon}|^{2})+\ln(1+|y|^{2})\right)-Cs.
\end{equation*}
From the fact that $(t'_{\varepsilon},s'_{\varepsilon},x'_{\varepsilon},y'_{\varepsilon})$ is the maximum point of $Q$, we get that $Q(t,s'_{\varepsilon},x,y'_{\varepsilon}):=\hat{U}(t,x)-Q_{1}(t,x)$ has a maximum point $(t'_{\varepsilon},x'_{\varepsilon})$, and $-Q(t'_{\varepsilon},s,x'_{\varepsilon},y):=\check{U}(s,y)-Q_{2}(s,y)$ has a minimum point $(s'_{\varepsilon},y'_{\varepsilon})$. Using the fact that $\hat{U}$ is a subsolution, we get
\begin{equation}\label{leq}
\frac{t'_{\varepsilon}-s'_{\varepsilon}}{\varepsilon}\leq H^{0}\left(x'_{\varepsilon},\frac{x'_{\varepsilon}-y'_{\varepsilon}}{\varepsilon}-\frac{2\gamma x'_{\varepsilon}}{1+|x'_{\varepsilon}|^{2}}\right).
\end{equation}
Since $\check{U}$ is a supersolution, we obtain
\begin{equation}\label{geq}
\frac{t'_{\varepsilon}-s'_{\varepsilon}}{\varepsilon}-C\geq H^{0}\left(y'_{\varepsilon},\frac{x'_{\varepsilon}-y'_{\varepsilon}}{\varepsilon}+\frac{2\gamma y'_{\varepsilon}}{1+|y'_{\varepsilon}|^{2}}\right).
\end{equation}
Subtracting \eqref{geq} from \eqref{leq}, we have
\begin{equation}\label{fin}
C\leq H^{0}\left(x'_{\varepsilon},\frac{x'_{\varepsilon}-y'_{\varepsilon}}{\varepsilon}-\frac{2\gamma x'_{\varepsilon}}{1+|x'_{\varepsilon}|^{2}}\right)-H^{0}\left(y'_{\varepsilon},\frac{x'_{\varepsilon}-y'_{\varepsilon}}{\varepsilon}+\frac{2\gamma y'_{\varepsilon}}{1+|y'_{\varepsilon}|^{2}}\right).
\end{equation}
Since $H^{0}$ is uniformly continuous over compact sets, and $|t'_{\varepsilon}-s'_{\varepsilon}|, |x'_{\varepsilon}-y'_{\varepsilon}|\rightarrow0$ as $\varepsilon\rightarrow0$, the right-hand side of \eqref{fin} goes to $0$ as $\varepsilon$, $\gamma$ go to $0$.
Notice that the terms $\frac{x'_{\varepsilon}-y'_{\varepsilon}}{\varepsilon}$, $\frac{2\gamma x'_{\varepsilon}}{1+|x'_{\varepsilon}|^{2}}$ and $\frac{2\gamma y'_{\varepsilon}}{1+|y'_{\varepsilon}|^{2}}$ are bounded and that $|x'_{\varepsilon}|, |y'_{\varepsilon}|\leq R_{\gamma}$ for each $\gamma$.
Taking $\varepsilon\rightarrow0$ and then $\gamma\rightarrow0$, we get $C\leq0$, reaching a contradiction.
\end{proof}

\begin{remark}
 When $H^{0}$ can be explicitly calculated, the continuity can be directly verified. Otherwise we may need to prove that the expression on the right-hand side of \eqref{fin} is non-positive by using the specifics for the case at hand.
\end{remark}

\begin{remark}
Here $H^{0}$ is uniformly continuous on compact sets, which is a sufficient condition for the comparison principle for \eqref{sU0} to hold; see \cite[Corollary 4]{KP17}. The comparison principle also holds if the coefficients $b_{2}(x,y)$, $\sigma_{2}(x,y)$, $k_{2}(x,y,z)$ are independent on $x$, the coefficient $\sigma_{1}(x,y)$ is bounded with respect to $y$ for each $x$, and the non-correlation condition $\rho=0$ holds.
\end{remark}

\subsection{The convergence result}\label{CR}

Next we will show that our problem satisfies the comparison principle in the condition of Lemma \ref{son}, this is, every subsolution of
\begin{equation*}
\partial_{t}U(t,x)\leq \hat{H}(x,p):=\inf_{\lambda\in\Lambda}\hat{H}^{\lambda}(x,p)=\inf_{\zeta\in C_{c}^{2}(\mathbb{R}),0<\theta<1}\hat{H}^{\zeta,\theta}(x,p),
\end{equation*}
where $\hat{H}^{\lambda}$ is as defined in \eqref{Hl1}, \eqref{H02} and \eqref{Hl3}, is less than or equal to every supersolution of
\begin{equation*}
\partial_{t}U(t,x)\geq \check{H}(x,p):=\sup_{\lambda\in\Lambda}\check{H}^{\lambda}(x,p)=\sup_{\zeta\in C_{c}^{2}(\mathbb{R}),0<\theta<1}\check{H}^{\zeta,\theta}(x,p),
\end{equation*}
where $\check{H}^{\lambda}$ is as defined in \eqref{Hm1}, \eqref{H22} and \eqref{Hm3}.

 An important step is to check that the following operator inequality
\begin{equation}\label{ineq}
\inf_{\lambda\in\Lambda}\hat{H}^{\lambda}(x,p)\leq H^{0}(x,p)\leq\sup_{\lambda\in\Lambda}\check{H}^{\lambda}(x,p)
\end{equation}
holds, where $H^{0}(x,p)$ is as defined in \eqref{super0U0}, \eqref{hh0} and \eqref{hHV}. With this inequality, we can use the arguments from the proof of the comparison principle presented in \cite{FK06}.

For the critical case ($\alpha=2$), we denote
\begin{equation*}
T(t)f(y):=\mathbb{E}[\,f(Y_{t}^{x})\,|\,Y_{0}^{x}=y\,],~~~f\in C_{b}(\mathbb{R}),
\end{equation*}
where $Y^{x}$ is the $\mathbb{R}$-valued Markov process as in \eqref{YYx}.
For each $f\in D^{++}(\mathcal{L}^{x,p})\subset C_{b}(\mathbb{R})$, using $(\textbf{C}_{5})-(\textbf{ii})$, there exists compact $K\subset\subset\mathbb{R}$ with
\begin{align}\nonumber
   &\sup\limits_{y\in\mathbb{R}}\{\,V^{x,p}(y)+(1-\theta)\frac{\mathcal{L}^{x,p}f(y)}{f(y)}+\theta e^{-\xi(y)}\mathcal{L}^{x,p}e^{\xi(y)}\,\}\\\label{Xian}
   &=\sup\limits_{y\in K}\{\,V^{x,p}(y)+(1-\theta)\frac{\mathcal{L}^{x,p}f(y)}{f(y)}+\theta e^{-\xi(y)}\mathcal{L}^{x,p}e^{\xi(y)}\,\}.
\end{align}
For each $\varepsilon>0$, by truncating and mollifying $f$, we can find a $\zeta\in C_{c}^{2}(\mathbb{R})$ such that
\begin{equation}\label{Bu}
\hat{H}^{\lambda}(x,p)\leq\varepsilon+\sup\limits_{y\in K}\{\,V^{x,p}(y)+(1-\theta)\frac{\mathcal{L}^{x,p}f(y)}{f(y)}+\theta e^{-\xi(y)}\mathcal{L}^{x,p}e^{\xi(y)}\,\},
\end{equation}
where $\hat{H}^{\lambda}(x,p)$ given by \eqref{H02}. Then by \eqref{Xian} and \eqref{Bu}, we have
\begin{equation}\label{infinf}
\inf_{\lambda\in\Lambda}\hat{H}^{\lambda}(x,p)\leq\inf_{0<\theta<1}\inf_{f\in D^{++}(\mathcal{L}^{x,p})}\sup\limits_{y\in\mathbb{R}}\{\,V^{x,p}(y)+(1-\theta)\frac{\mathcal{L}^{x,p}f(y)}{f(y)}+\theta e^{-\xi(y)}\mathcal{L}^{x,p}e^{\xi(y)}\,\}.
\end{equation}
Similarly, we obtain
\begin{equation}\label{supsup}
\sup_{\lambda\in\Lambda}\check{H}^{\lambda}(x,p)\geq\sup_{0<\theta<1}\sup_{f\in D^{++}(\mathcal{L}^{x,p})}\inf\limits_{y\in\mathbb{R}}\{\,V^{x,p}(y)+(1+\theta)\frac{\mathcal{L}^{x,p}f(y)}{f(y)}-\theta e^{-\xi(y)}\mathcal{L}^{x,p}e^{\xi(y)}\,\},
\end{equation}
where $\check{H}^{\lambda}(x,p)$ as in \eqref{H22}.

We can find a sequence $\{f_{n}\}\subset D^{++}(\mathcal{L}^{x,p})$ by taking $f_{n}:=e^{\xi_{n}}$, where $\xi_{n}\in C_{c}^{2}(\mathbb{R})$ are some smooth truncations of $\xi$, such that
\begin{equation*}
\int_{\mathbb{R}}e^{-\xi(y)}\mathcal{L}^{x,p}e^{\xi(y)}d\mu(y)\geq\limsup_{n\rightarrow\infty}\int_{\mathbb{R}}\frac{\mathcal{L}^{x,p}f_{n}(y)}{f_{n}(y)}d\mu(y),\quad\mu\in\mathcal{P}(\mathbb{R}).
\end{equation*}
Then we get
{\small\begin{equation*}
\inf_{f\in D^{++}(\mathcal{L}^{x,p})}\int_{\mathbb{R}}\frac{\mathcal{L}^{x,p}f(y)}{f(y)}d\mu(y)\wedge\int_{\mathbb{R}}e^{-\xi(y)}\mathcal{L}^{x,p}e^{\xi(y)}d\mu(y)=\inf\limits_{f\in D^{++}(\mathcal{L}^{x,p})}\int_{\mathbb{R}}\frac{\mathcal{L}^{x,p}f(y)}{f(y)}d\mu(y)=-J^{x,p}(\mu),
\end{equation*}}
where $J^{x,p}(\mu)$ defined in \eqref{j}.

Let $f_{1},\cdot\cdot\cdot,f_{m}\in D^{++}(\mathcal{L}^{x,p})$, $\sum_{i=1}^{m}\alpha_{i}=1$, $\alpha_{i}\geq0$, and set
\begin{equation*}
f_{h}=\frac{1}{h}\int_{0}^{h}T(s)\prod_{i=1}^{m}f_{i}^{\alpha_{i}}ds,\quad\text{for}~~h>0.
\end{equation*}
Observing that
\begin{equation*}
T(t)\prod_{i=1}^{m}f_{i}^{\alpha_{i}}\leq\prod_{i=1}^{m}(T(t)f_{i})^{\alpha_{i}}.
\end{equation*}
As $h\rightarrow0$, we have
\begin{equation*}
\mathcal{L}^{x,p}f_{h}=\frac{T(h)f_{h}-f_{h}}{h}\leq\frac{1}{h}\left(\prod_{i=1}^{m}(T(h)f_{i})^{\alpha_{i}}-\prod_{i=1}^{m}f_{i}^{\alpha_{i}}\right)\rightarrow\prod_{i=1}^{m}f_{i}^{\alpha_{i}}\sum_{i=1}^{m}\alpha_{i}\frac{\mathcal{L}^{x,p}f_{i}}{f_{i}},
\end{equation*}
where the convergence is uniform. It follows that
{\footnotesize\begin{align*}
&\inf_{0<\theta<1}\inf_{f\in D^{++}(\mathcal{L}^{x,p})}\sup\limits_{y\in\mathbb{R}}\{\,V^{x,p}(y)+(1-\theta)\frac{\mathcal{L}^{x,p}f(y)}{f(y)}+\theta e^{-\xi(y)}\mathcal{L}^{x,p}e^{\xi(y)}\,\}\\
&\leq\inf_{0<\theta<1}\inf_{\sum_{i=1}^{m}\alpha_{i}=1}\inf_{f_{i}\in D^{++}(\mathcal{L}^{x,p})}\sup\limits_{y\in\mathbb{R}}\{\,V^{x,p}(y)+(1-\theta)\sum_{i=1}^{m}\alpha_{i}\frac{\mathcal{L}^{x,p}f_{i}(y)}{f_{i}(y)}+\theta e^{-\xi(y)}\mathcal{L}^{x,p}e^{\xi(y)}\,\}\\
&=\lim_{m\rightarrow\infty}\inf_{0<\theta<1}\inf_{\sum_{i=1}^{m}\alpha_{i}=1}\sup\limits_{y\in\mathbb{R}}\{\,V^{x,p}(y)+(1-\theta)\sum_{i=1}^{m}\alpha_{i}\frac{\mathcal{L}^{x,p}f_{i}(y)}{f_{i}(y)}+\theta e^{-\xi(y)}\mathcal{L}^{x,p}e^{\xi(y)}\,\}\\
&=\lim_{m\rightarrow\infty}\inf_{0<\theta<1}\sup\limits_{\mu\in\mathcal{P}(\mathbb{R})}\inf_{\sum_{i=1}^{m}\alpha_{i}=1}\int_{\mathbb{R}}\Big(V^{x,p}(y)+(1-\theta)\sum_{i=1}^{m}\alpha_{i}\frac{\mathcal{L}^{x,p}f_{i}(y)}{f_{i}(y)}+\theta e^{-\xi(y)}\mathcal{L}^{x,p}e^{\xi(y)}\Big)\mu(dy)\\
&=\lim_{m\rightarrow\infty}\inf_{0<\theta<1}\sup\limits_{\mu\in\mathcal{P}(\mathbb{R})}\Big[\int_{\mathbb{R}}V^{x,p}(y)\mu(dy)+(1-\theta)\min_{1\leq i\leq m}\int_{\mathbb{R}}\frac{\mathcal{L}^{x,p}f_{i}(y)}{f_{i}(y)}\mu(dy)\wedge\int_{\mathbb{R}}e^{-\xi(y)}\mathcal{L}^{x,p}e^{\xi(y)}\mu(dy)\\
&\,\,\,\,\,\,\,\,\,+\theta\int_{\mathbb{R}}e^{-\xi(y)}\mathcal{L}^{x,p}e^{\xi(y)}\mu(dy)\Big]\\
&=\lim_{m\rightarrow\infty}\sup\limits_{\mu\in\mathcal{P}(\mathbb{R})}\Big[\int_{\mathbb{R}}V^{x,p}(y)\mu(dy)+\min_{1\leq i\leq m}\int_{\mathbb{R}}\frac{\mathcal{L}^{x,p}f_{i}(y)}{f_{i}(y)}\mu(dy)\wedge\int_{\mathbb{R}}e^{-\xi(y)}\mathcal{L}^{x,p}e^{\xi(y)}\mu(dy)\Big]\\
&=\sup_{\mu\in\mathcal{P}(\mathbb{R})}\big[\int_{\mathbb{R}}V^{x,p}(y)d\mu(y)-J^{x,p}(\mu)\big]=H^{0}(x,p).
\end{align*}}
Combined with \eqref{infinf}, we obtain
\begin{equation*}
\inf_{\lambda\in\Lambda}\hat{H}^{\lambda}(x,p)\leq H^{0}(x,p).
\end{equation*}

Based on \eqref{supsup} and \cite[Lemma B.10]{FK06}, we have
\begin{equation}\label{bus}
\sup_{\lambda\in\Lambda}\check{H}^{\lambda}(x,p)\geq\inf_{\nu\in\mathcal{P}(\mathbb{R})}\liminf_{t\rightarrow\infty}\frac{1}{t}\ln \mathbb{E}^{\nu}[e^{\int_{0}^{t}V^{x,p}(Y_{s}^{x,p})ds}].
\end{equation}
Thus in order to prove that
\begin{equation*}
 H^{0}(x,p)\leq\sup_{\lambda\in\Lambda}\check{H}^{\lambda}(x,p),
\end{equation*}
it suffices to check that if
\begin{equation*}
\inf_{\nu\in\mathcal{P}(\mathbb{R})}\liminf_{t\rightarrow\infty}\frac{1}{t}\ln \mathbb{E}^{\nu}\left[e^{\int_{0}^{t}V^{x,p}(Y_{s}^{x,p})ds}\right]\geq H^{0}(x,p).
\end{equation*}
Define the occupation measures of the process $Y^{x,p}$:
\begin{equation*}
\mu_{t}^{x,p}(\cdot)=\frac{1}{t}\int_{0}^{t}1_{Y_{s}^{x,p}}(\cdot)ds\quad\text{for}~~t>0.
\end{equation*}
Since the Euclidean metric space $(\mathbb{R}, d)$ is separable, the Prokhorov metric space $(\mathcal{P}(\mathbb{R}), \tilde{d})$ is also separable. And then convergence of measures in the Prokhorov metric is equivalent to weak convergence of measures. For $B\in\mathcal{B}(\mathcal{P}(\mathbb{R}))$,
\begin{equation*}
\mathbb{P}_{t}^{y_{0}}(B):=\mathbb{P}(\mu_{t}^{x,p}(\cdot)\in B|Y_{0}^{x,p}=y_{0})
\end{equation*}
is a probability measure on $\mathcal{P}(\mathbb{R})$ induced by the occupation measures $\mu_{t}^{x,p}$ of the process $Y^{x,p}$ with initial value  $Y_{0}^{x,p}=y_{0}$.

Define $\phi: \mathcal{P}(\mathbb{R})\rightarrow\mathbb{R}$ by $\phi(\mu)=\int_{\mathbb{R}}V^{x,p}(y)\mu(dy)$. Let $V_{\varepsilon}^{x,p}:=V^{x,p}\mathds{1}_{\{V^{x,p}\leq \varepsilon\}}+\varepsilon\mathds{1}_{\{V^{x,p}\geq \varepsilon\}}$ for $\varepsilon\geq\inf_{y\in\mathbb{R}}V^{x,p}(y)$, then $V_{\varepsilon}^{x,p}$ is bounded.
By definition of weak convergence of measures, if $\mu_{n}\rightarrow\mu$ weakly, we have
\begin{equation*}
\lim_{n\rightarrow\infty}\int_{\mathbb{R}}V_{\varepsilon}^{x,p}(y)\mu_{n}(dy)=\int_{\mathbb{R}}V_{\varepsilon}^{x,p}(y)\mu(dy).
\end{equation*}
Furthermore, using the monotone convergence theorem, we get
\begin{align*}
\phi(\mu)&=\lim_{\varepsilon\rightarrow\infty}\int_{\mathbb{R}}V_{\varepsilon}^{x,p}(y)\mu(dy)=\lim_{\varepsilon\rightarrow\infty}\lim_{n\rightarrow\infty}\int_{\mathbb{R}} V_{\varepsilon}^{x,p}(y)\mu_{n}(dy)\\
         &=\sup_{\varepsilon}\liminf_{n\rightarrow\infty}\int_{\mathbb{R}} V_{\varepsilon}^{x,p}(y)\mu_{n}(dy)\leq\liminf_{n\rightarrow\infty}\sup_{\varepsilon}\int_{\mathbb{R}}V_{\varepsilon}^{x,p}(y)\mu_{n}(dy)\\
         &=\liminf_{n\rightarrow\infty}\int_{\mathbb{R}}V^{x,p}(y)\mu_{n}(dy)=\liminf_{n\rightarrow\infty}\phi(\mu_{n}),
\end{align*}
which shows that $\phi(\mu)$ is a lower semicontinuous function on $\mathcal{P}(\mathbb{R})$.

Fix $\nu\in\mathcal{P}(\mathbb{R})$, then there exists a compact set $K$ in $\mathbb{R}$ such that $\nu(K)>0$.
For $\mu\in\mathcal{P}(\mathbb{R})$, $B_{r}(\mu)$ is the open ball in $\mathcal{P}(\mathbb{R})$ of radius $r$ centered at $\mu$. Then we have
{\small\begin{align}\nonumber
\liminf_{t\rightarrow\infty}\frac{1}{t}\ln \mathbb{E}^{\nu}\left[e^{\int_{0}^{t}V^{x,p}(Y_{s}^{x,p})ds}\right]
&=\liminf_{t\rightarrow\infty}\frac{1}{t}\ln \mathbb{E}^{\nu}\left[e^{t\phi(\mu_{t}^{x,p})}\right]\\ \nonumber
&\geq\liminf_{t\rightarrow\infty}\frac{1}{t}\ln \mathbb{E}^{\nu}\left[e^{t\phi(\mu_{t}^{x,p})}1_{\{Y_{0}^{x,p}\in K\}}\right]\\ \nonumber
&\geq\liminf_{t\rightarrow\infty}\frac{1}{t}\ln\left[\inf_{y_{0}\in K}\mathbb{E}^{y_{0}}(e^{t\phi(\mu_{t}^{x,p})})\right]+\liminf_{t\rightarrow\infty}\frac{1}{t}\ln\nu(K)\\ \nonumber
&=\liminf_{t\rightarrow\infty}\frac{1}{t}\ln\left[\inf_{y_{0}\in K}\int_{\tilde{\mu}\in\mathcal{P}(\mathbb{R})}e^{t\phi(\tilde{\mu})}d\mathbb{P}_{t}^{y_{0}}(\tilde{\mu})\right]\\ \nonumber
&\geq\liminf_{t\rightarrow\infty}\frac{1}{t}\ln\left[\inf_{y_{0}\in K}\int_{\tilde{\mu}\in B_{r}(\mu)}e^{t\phi(\tilde{\mu})}d\mathbb{P}_{t}^{y_{0}}(\tilde{\mu})\right]\\ \nonumber
&\geq\inf_{\tilde{\mu}\in B_{r}(\mu)}\phi(\tilde{\mu})+\liminf_{t\rightarrow\infty}\frac{1}{t}\ln\left[\inf_{y_{0}\in K}\mathbb{P}_{t}^{y_{0}}(B_{r}(\mu))\right]\\\label{low}
&\geq\inf_{\tilde{\mu}\in B_{r}(\mu)}\phi(\tilde{\mu})-J^{x,p}(\mu).
\end{align}}
The last inequality follows from the uniform lower bound of large deviation principle for the occupation measures $\mu_{t}^{x,p}$:
\begin{equation*}
\liminf_{t\rightarrow\infty}\frac{1}{t}\ln[\inf_{y_{0}\in K}\mathbb{P}_{t}^{y_{0}}(B_{r}(\mu))]\geq-J^{x,p}(\mu),
\end{equation*}
which can be obtained from \cite[Theorem 5.5]{DV83} under the condition ($\textbf{C}_{4}$) by the contraction principle.

Because $\phi$ is a lower semicontinuous function, by taking limit \eqref{low} as $r\rightarrow0$, we get
\begin{equation*}
\liminf_{t\rightarrow\infty}\frac{1}{t}\ln \mathbb{E}^{\nu}\left[e^{\int_{0}^{t}V^{x,p}(Y_{s}^{x,p})ds}\right]\geq\phi(\mu)-J^{x,p}(\mu).
\end{equation*}
Since $\mu$ is arbitrary, moreover, we have
\begin{equation*}
\inf_{\nu\in\mathcal{P}(\mathbb{R})}\liminf_{t\rightarrow\infty}\frac{1}{t}\ln \mathbb{E}^{\nu}\left[e^{\int_{0}^{t}V^{x,p}(Y_{s}^{x,p})ds}\right]\geq\sup_{\mu\in\mathcal{P}(\mathbb{R})}\{\phi(\mu)-J^{x,p}(\mu)\}=H^{0}(x,p).
\end{equation*}
Then we finished the proof of the operator inequality \eqref{ineq} in critical case. For the supercritical case ($\alpha>2$) and subcritical case ($\alpha<2$), the operator inequality \eqref{ineq} can be proven with similar ideas.

\medskip
 Now we state the main result of the paper, namely, the convergence result for the Cauchy problem for partial integro-differential equations.
\begin{lemma}\label{con}
Assumptions $(\textbf{C}_{1})-(\textbf{C}_{5})$ hold. Suppose the comparison principle holds for the Cauchy problem \eqref{sU0}.
Then the sequence of functions $\{U^{\varepsilon}\}_{\varepsilon>0}$ defined in the Cauchy problem \eqref{sU} converge uniformly over compact subsets of $[0,T]\times\mathbb{R}\times\mathbb{R}$ as $\varepsilon\rightarrow0$ to the unique continuous viscosity solution $U^{0}$ of \eqref{sU0}.
\end{lemma}
\begin{proof}
For $\{U^{\varepsilon}\}_{\varepsilon>0}$ defined in the Cauchy problem \eqref{sU}, we can get the corresponding $\hat{U}$ and $\check{U}$ by using Definition \ref{defi}. It has been checked that Condition \ref{cond1} holds for $H^{\varepsilon}$ defined by \eqref{Hed}. Applying Lemma \ref{subsup}, $\hat{U}$ is a subsolution and $\check{U}$ is a supersolution of the Cauchy problem \eqref{sU0}. Based on Theorem \ref{unf}, the comparison principle is valid for the Cauchy problem \eqref{sU0}. Then from Lemma \ref{son}, we have $\hat{U}=\check{U}$ and  $U^{\varepsilon}\rightarrow U^{0}:=\hat{U}=\check{U}$  uniformly over compact subsets of $[0,T]\times\mathbb{R}\times\mathbb{R}$.
\end{proof}

\section{Large deviation principle}\label{secton5}

The Bryc's theorem (see \cite[p142]{DZ10}) permits to derive the large deviation principle as a consequence of exponential tightness of $\{X_{t}^{\varepsilon}\}_{\varepsilon>0}$ in system \eqref{2} and the existence of $U^{0}$ based on \eqref{sU0} for every $h\in C_{b}(\mathbb{R})$. We begin with the following lemma.
\begin{lemma}\label{exp0}
The sequence of processes $\{X_{t}^{\varepsilon}\}_{\varepsilon>0}$ in system \eqref{2} is exponentially tight.
\end{lemma}
\begin{proof}
Define
\begin{eqnarray*}
      f_{\varepsilon}(x,y):=
        \left\{\begin{array}{l}
       f(x)+\varepsilon^{\alpha-1}\xi(y)\quad\text{for}~~\alpha\geq2;\\
       f(x)+\varepsilon^{\frac{\alpha}{2}}\xi(y)\quad\text{for}~~1<\alpha<2.
        \end{array}\right.
      \end{eqnarray*}
Where $f(x):=\ln(1+x^{2})$ and $\xi(y)$ is the positive Lyapunov function satisfying condition ($\textbf{C}_{5}$) with $\theta=1$. Since $f(x)$ is an increasing function of $|x|$ and $\xi(y)>0$, we have that for any $c>0$ there exits a compact set $K_{c}\subset\mathbb{R}$ such that
\begin{equation}\label{c}
f_{\varepsilon}(x,y)>0 \quad\text{for}~~ x\notin K_{c},~y\in\mathbb{R}.
\end{equation}
When $\alpha>2$,  we have
\begin{align*}
H^{\varepsilon}f_{\varepsilon}(x,y)=&\sigma_{1}^{2}(x,y)\big(\partial_{x}f(x)\big)^{2}+\varepsilon\big(b_{1}(x,y)\partial_{x}f(x)+\sigma_{1}^{2}(x,y)\partial_{xx}^{2}f(x)\big)\\ \nonumber
&+\int_{\mathbb{R}\setminus \{0\}}\Big(e^{\frac{1}{\varepsilon}[f(x+\varepsilon k_{1}(x,y,z))-f(x)]}-1-k_{1}(x,y,z)\partial_{x}f(x)\Big)\nu_{1}(dz)\\ \nonumber
&+2\varepsilon^{\frac{\alpha}{2}-1}\rho\sigma_{1}(x,y)\sigma_{2}(x,y)\partial_{x}f(x)\partial_{y}\xi(y) +\varepsilon^{2-\alpha}e^{-\varepsilon^{\alpha-2}\xi(y)}\mathcal{L}^{x}e^{\varepsilon^{\alpha-2}\xi(y)}.
\end{align*}
For $\alpha=2$, we get
\begin{align*}
H^{\varepsilon}f_{\varepsilon}(x,y)=&\sigma_{1}^{2}(x,y)\big(\partial_{x}f(x)\big)^{2}+\varepsilon \big(b_{1}(x,y)\partial_{x}f(x)+\sigma_{1}^{2}(x,y)\partial_{xx}^{2}f(x)\big)\\ \nonumber
&+\int_{\mathbb{R}\setminus \{0\}}\Big(e^{\frac{1}{\varepsilon}[f(x+\varepsilon k_{1}(x,y,z))-f(x)]}-1-k_{1}(x,y,z)\partial_{x}f(x)\Big)\nu_{1}(dz)\\ \nonumber
&+2\rho\sigma_{1}(x,y)\sigma_{2}(x,y)\partial_{x}f(x)\partial_{y}\xi(y)+e^{-\xi(y)}\mathcal{L}^{x}e^{\xi(y)}.
\end{align*}
In the case $1<\alpha<2$, we obtain
\begin{align*}
H^{\varepsilon}f_{\varepsilon}(x,y)=&\sigma_{1}^{2}(x,y)\big(\partial_{x}f(x)\big)^{2}+\varepsilon\big(b_{1}(x,y)\partial_{x}f(x)+\sigma_{1}^{2}(x,y)\partial_{xx}^{2}f(x)\big)\\ \nonumber
&+\int_{\mathbb{R}\setminus \{0\}}\Big(e^{\frac{1}{\varepsilon}[f(x+\varepsilon k_{1}(x,y,z))-f(x)]}-1-k_{1}(x,y,z)\partial_{x}f(x)\Big)\nu_{1}(dz)\\ \nonumber
&+2\rho\sigma_{1}(x,y)\sigma_{2}(x,y)\partial_{x}f(x)\partial_{y}\xi(y)+\varepsilon^{2-\alpha}e^{-\varepsilon^{\frac{\alpha}{2}-1}\xi(y)}\mathcal{L}^{x}e^{\varepsilon^{\frac{\alpha}{2}-1}\xi(y)}.                           \end{align*}
We observe that $|\partial_{x}f(x)|+|\partial_{xx}^{2}f(x)|<\infty$ and growth condition on the coefficients. By our choice of $\xi$ with bounded first and second derivatives, there exists $C>0$ such that
\begin{equation}\label{C}
\sup_{x\in\mathbb{R},y\in\mathbb{R}}H^{\varepsilon}f_{\varepsilon}(x,y)\leq C<\infty, \quad\forall~\varepsilon>0.
\end{equation}
The $\mathbb{P}$ and $\mathbb{E}$ in the following proof denote probability and expectation conditioned on $(X_{t}^{\varepsilon},Y_{t}^{\varepsilon})$ starting at $(x,y)$. Define the process
\begin{equation*}
M_{t}^{\varepsilon}=\exp\left\{\frac{f_{\varepsilon}(X_{t}^{\varepsilon},Y_{t}^{\varepsilon})}{\varepsilon}-\frac{f_{\varepsilon}(x,y)}{\varepsilon}-\frac{1}{\varepsilon}\int_{0}^{t}H^{\varepsilon}f_{\varepsilon}(X_{s}^{\varepsilon},Y_{s}^{\varepsilon})ds\right\}
\end{equation*}
Then $M_{t}^{\varepsilon}$ is a supermatingale and hence we can apply the optional sampling theorem, that is $\mathbb{E}[M_{t}^{\varepsilon}]\leq1$. So
\begin{equation*}
 1\geq \mathbb{E}[M_{t}^{\varepsilon}|X_{t}^{\varepsilon}\notin K_{c}]\geq \mathbb{E}[e^{\frac{(c-f_{\varepsilon}(x,y)-tC)}{\varepsilon}}|X_{t}^{\varepsilon}\notin K_{c}]=\mathbb{P}(X_{t}^{\varepsilon}\notin K_{c})e^{\frac{(c-f_{\varepsilon}(x,y)-tC)}{\varepsilon}},
\end{equation*}
where we have used \eqref{c} and \eqref{C} to estimate the first and third term in $M_{t}^{\varepsilon}$. Then we get
\begin{equation*}
\varepsilon\ln\mathbb{P}(X_{t}^{\varepsilon}\notin K_{c})\leq tC+f_{\varepsilon}(x,y)-c
\end{equation*}
and this finally gives us the exponential tightness of $\{X_{t}^{\varepsilon}\}_{\varepsilon>0}$.
\end{proof}

\medskip

 We now proceed to argue that a large deviation principle holds for $\{X_{t}^{\varepsilon}\}_{\varepsilon>0}$ as $\varepsilon\rightarrow0$ with speed $1/\varepsilon$ and good rate function $I$ .
\begin{theorem}
Let $X_{0}^{\varepsilon}=x_{0}$, and suppose conditions $(\textbf{C}_{1})-(\textbf{C}_{5})$ hold. Then $\{X_{t}^{\varepsilon}\}_{\varepsilon>0}$ in system \eqref{2} satisfies a large deviation principle in $\mathbb{R}$ with respect to $\mathcal{B}(\mathbb{R})$ with good rate function
\begin{equation}\label{I0}
I(x,x_{0},t)=\sup_{h\in C_{b}(\mathbb{R})}\{h(x)-U^{0}(t,x_{0})\}.
\end{equation}
\end{theorem}
\begin{remark}
If the coefficients $b_{1}(x,y)$, $\sigma_{1}(x,y)$, $k_{1}(x,y,z)$, $b_{2}(x,y)$, $\sigma_{2}(x,y)$, $k_{2}(x,y,z)$ are independent of $x$, then $H^{0}(x,p)$ becomes $H^{0}(p)$. Using \cite[Lemma D.1]{FFK12}, we obtain
\begin{equation}\label{I10}
I(x,x_{0},t)=tQ^{0}\big(\frac{x_{0}-x}{t}\big),
\end{equation}
where $Q^{0}(\cdot)$ is the Legendre transform of the convex function $H^{0}(\cdot)$.
\end{remark}

\begin{example}
We consider a family of models of the form
\begin{eqnarray}\label{Levy}
\left\{\begin{array}{l}
 dX^{\varepsilon,\delta}_{t}=\varepsilon b_{1}(X^{\varepsilon,\delta}_{t-},Y^{\varepsilon,\delta}_{t-})dt+\sqrt{2\varepsilon}\sigma_{1}(X^{\varepsilon,\delta}_{t-},Y^{\varepsilon,\delta}_{t-})dW_{t}^{(1)}\\~~~~~~~~~~~~~+\varepsilon k_{1}(X^{\varepsilon,\delta}_{t-},Y^{\varepsilon,\delta}_{t-},z)dL_{t}^{\alpha_{1},\frac{1}{\varepsilon}},\quad X^{\varepsilon,\delta}_{0}=x_{0}\in\mathbb{R},  \\[0.5ex]
dY^{\varepsilon,\delta}_{t}=\frac{\varepsilon}{\delta}b_{2}(X^{\varepsilon,\delta}_{t-},Y^{\varepsilon,\delta}_{t-})dt+\sqrt{\frac{2\varepsilon}{\delta}}
\sigma_{2}(X^{\varepsilon,\delta}_{t-},Y^{\varepsilon,\delta}_{t-})\big(\rho dW_{t}^{(1)}+\sqrt{1-\rho^{2}}dW_{t}^{(2)}\big)\\~~~~~~~~~~~~~+ k_{2}(X^{\varepsilon,\delta}_{t-},Y^{\varepsilon,\delta}_{t-},z)dL_{t}^{\alpha_{2},\frac{\varepsilon}{\delta}},\quad Y^{\varepsilon,\delta}_{0}=y_{0}\in\mathbb{R},
 \end{array}\right.
\end{eqnarray}
where parameters $\varepsilon$, $\delta>0$ and $\rho\in(-1,1)$, and coefficients $b_{1}(x,y)$, $b_{2}(x,y)$, $\sigma_{1}(x,y)$, $\sigma_{2}(x,y)$, $k_{1}(x,y,z)$, $k_{2}(x,y,z)$ with $(x,y)\in\mathbb{R}^{2}$ and $|z|>0$, as noted previously.
 $W^{(1)}, W^{(2)}$ are independent Brownian motions independent of two independent $\alpha$-stable L\'evy processes $L_{t}^{\alpha_{1},\frac{1}{\varepsilon}}$, $L_{t}^{\alpha_{2},\frac{\varepsilon}{\delta}} (1<\alpha_{1}, \alpha_{2}<2)$. Here
$L_{t}^{\alpha_{1},\frac{1}{\varepsilon}}\sim\nu_{1}(dz)=\frac{1}{\varepsilon}\frac{1}{|z|^{1+\alpha_{1}}}dz$,
$L_{t}^{\alpha_{2},\frac{\varepsilon}{\delta}}\sim\nu_{2}(dz)=\frac{\varepsilon}{\delta}\frac{1}{|z|^{1+\alpha_{2}}}dz$.

The L\'evy-It$\hat{o}$ decompositions for $L_{t}^{\alpha_{1},\frac{1}{\varepsilon}}$, $L_{t}^{\alpha_{2},\frac{\varepsilon}{\delta}}$ (see \cite{D15,S15}) are
{\footnotesize\begin{align*}
  L_{t}^{\alpha_{1},\frac{1}{\varepsilon}}&=\int_{0<|z|<1}z\tilde{N}^{(1),\frac{1}{\varepsilon}}(dz,t)+\int_{|z|\geq1}zN^{(1),\frac{1}{\varepsilon}}(dz,t)=\int_{\mathbb{R}\setminus \{0\}}z\tilde{N}^{(1),\frac{1}{\varepsilon}}(dz,t)-\frac{1}{\varepsilon}\int_{|z|\geq1}z\nu_{1}(dz),\\
L_{t}^{\alpha_{2},\frac{\varepsilon}{\delta}}  &=\int_{0<|z|<1}z\tilde{N}^{(2),\frac{\varepsilon}{\delta}}(dz,t)+\int_{|z|\geq1}zN^{(2),\frac{\varepsilon}{\delta}}(dz,t)=\int_{\mathbb{R}\setminus \{0\}}z\tilde{N}^{(2),\frac{\varepsilon}{\delta}}(dz,t)-\frac{\varepsilon}{\delta}\int_{|z|\geq1}z\nu_{2}(dz).
\end{align*}}
Since $K_{\nu_{1}}:=\frac{1}{\varepsilon}\int_{|z|\geq1}z\nu_{1}(dz)<\infty$, $K_{\nu_{2}}:=\int_{|z|\geq1}z\nu_{2}(dz)<\infty$, the system \eqref{Levy} can be rewritten into
{\small\begin{eqnarray}\label{llevy}
\left\{\begin{array}{l}
 dX^{\varepsilon,\delta}_{t}=\varepsilon (b_{1}(X^{\varepsilon,\delta}_{t-},Y^{\varepsilon,\delta}_{t-})+K_{\nu_{1}})dt+\sqrt{2\varepsilon}\sigma_{1}(X^{\varepsilon,\delta}_{t-},Y^{\varepsilon,\delta}_{t-})dW_{t}^{(1)}\\~~~~~~~~~+\varepsilon\int_{\mathbb{R}\setminus \{0\}} k_{1}(X^{\varepsilon,\delta}_{t-},Y^{\varepsilon,\delta}_{t-},z)\tilde{N}^{(1),\frac{1}{\varepsilon}}(dz,dt),\quad X^{\varepsilon,\delta}_{0}=x_{0}\in\mathbb{R},  \\[0.5ex]
dY^{\varepsilon,\delta}_{t}=\frac{\varepsilon}{\delta}(b_{2}(X^{\varepsilon,\delta}_{t-},Y^{\varepsilon,\delta}_{t-})+K_{\nu_{2}})dt+\sqrt{\frac{2\varepsilon}{\delta}}
\sigma_{2}(X^{\varepsilon,\delta}_{t-},Y^{\varepsilon,\delta}_{t-})\big(\rho dW_{t}^{(1)}+\sqrt{1-\rho^{2}}dW_{t}^{(2)}\big)\\~~~~~~~~~+\int_{\mathbb{R}\setminus \{0\}} k_{2}(X^{\varepsilon,\delta}_{t-},Y^{\varepsilon,\delta}_{t-},z)\tilde{N}^{(2),\frac{\varepsilon}{\delta}}(dz,dt),\quad Y^{\varepsilon,\delta}_{0}=y_{0}\in\mathbb{R}.
 \end{array}\right.
\end{eqnarray}}
In order to prove the large deviation principle for the slow variables $\{X_{t}^{\varepsilon,\delta}\}_{\varepsilon,\delta>0}$ of system \eqref{Levy},
we can alternatively check that if the system \eqref{llevy} satisfies a large deviation principle. The proof can be illustrated with a series of steps that are similar to that of system \eqref{1}.
\end{example}

\section{Conclusions and future challenges}\label{CFC}
In the present work, we analysed the nonlinear slow-fast stochastic dynamical system \eqref{1} driven by both Brownian noises and
L\'evy noises. We characterised detailedly the limit Hamiltonian $H^{0}$ that has different forms in the three regimes depending on $\alpha>1$, i.e., supercritical case for $\alpha>2$, critical case for $\alpha=2$ and subcritical case for $\alpha<2$. We established the comparison principle, and  verified that the solutions of the Cauchy problem \eqref{sU} with the Hamiltonian $H^{\varepsilon}$ converge to the unique viscosity solution of the Cauchy problem \eqref{sU0} with the limit Hamiltonian $H^{0}$. Moreover, and arguably more importantly,
we derived the large deviation principle for the slow variables $\{X_{t}^{\varepsilon}\}_{\varepsilon>0}$ with an application example.

Since the standard Brownian motion and compound Poisson process are special L\'evy processes belonging to square-integrable martingales admitting second moments, we restrict ourselves to the large deviations with good rate function characterizing the Legendre transform of the log moment generating function. This can be used to provide a further description of the structure of the large deviation principle. However, L\'evy processes are not always integrable where the moment generating function does not exist. Generally, any L\'evy process is a semimartingale. It is interesting and challenging to create new methods computing the probability of a large deviation event for the slow variables in the nonlinear slow-fast stochastic dynamical systems.
\bigskip

\noindent\textbf{ACKNOWLEDGMENTS}

The authors are happy to thank Franziska K\"{u}hn for fruitful discussions on large deviations for L\'evy processes.



\begin{thebibliography}{00}
\bibitem{Ap04}
D. Applebaum, L\'evy Processes and Stochastic Calculus, Cambridge University Press, 2009.

\bibitem{Aw91}
S. Awatif, Equqtions D'Hamilton-Jacobi Du Premier Ordre Avec Termes Int\'egro-Diff\'erentiels, Commun. Partial Differ. Equ. 16 (1991) 1057-1074.

\bibitem{AT16}
O. Alvarez, A. Tourin, Viscosity solutions of nonlinear integro-differential equations, Annales de l'IHP Analyse non lin\'eaire 13 (2016) 293-317.

\bibitem{Ba03}
V.I. Bakhtin, Cram\'er's asymptotics in systems with fast and slow motions, Stoch. Stoch. Rep. 75 (2003) 319-341.

\bibitem{Bc08}
M. Bardi, I. Capuzzo Dolcetta, Optimal control and viscosity solutions of Hamilton-Jacobi-Bellman equations, Springer Science and Business Media, 2008.

\bibitem{BI08}
G. Barles, C. Imbert, Second-order elliptic integro-differential equations: viscosity solutions' theory revisited, Annales de l'IHP Analyse non lin\'eaire 25 (2008) 567-585.

\bibitem{BCG15}
M. Bardi, A. Cesaroni, D. Ghilli, Large deviations for some fast stochastic volatility models by viscosity methods, DCDS-A 35 (2015) 3965-3988.

\bibitem{BCS16}
M. Bardi, A. Cesaroni, A. Scotti, Convergence in multiscale financial models with non-Gaussian stochastic volatility,	ESAIM Control Optim. Calc. Var. 22 (2016) 500-518.

\bibitem{Bo16}
F. Bouchet, T. Grafke, T. Tangarife, E. Vanden-Eijnden, Large deviations in fast-slow systems, J. Stat. Phys. 162 (2016) 793-812.

\bibitem{D15}
J. Duan, An introduction to stochastic dynamics, Cambridge University Press, 2015.

\bibitem{DZ10}
A. Dembo, O. Zeitouni, Large deviations techniques and applications, Springer Science and Business Media, 2009.

\bibitem{DV75}
M.D. Donsker, S.R.S. Varadhan, On a variational formula for the principal eigenvalue for operators with maximum principle, PNAS 72 (1975) 780-783.

\bibitem{DV83}
M.D. Donsker, S.R.S. Varadhan, Asymptotic evaluation of certain Markov process expectations for large time--IV, Commun. Pure Appl. Math. 36 (1983) 183-212.

\bibitem{FK06}
J. Feng, T.G. Kurtz, Large deviations for stochastic processes, American Mathematical Soc., 2006.

\bibitem{FS06}
W.H. Fleming, H.M. Soner, Controlled Markov processes and viscosity solutions, Springer Science and Business Media, 2006.

\bibitem{FW12}
M.I. Freidlin, A.D. Wentzell, Random Perturbations of Dynamical Systems, Springer, 2012.

\bibitem{FFK12}
J. Feng, J.P. Fouque, R. Kumar, Small-time asymptotics for fast mean-reverting stochastic volatility models, Ann. Appl. Probab. 22 (2012) 1541-1575.

\bibitem{DH08}
F.D. Hollander, Large deviations, American Mathematical Soc., 2008.

\bibitem{JH16}
J. Hollender, L\'evy-Type Processes under Uncertainty and Related Nonlocal Equations, 2016.

\bibitem{JK06}
E.R. Jakobsen, K.H. Karlsen, A ``maximum principle for semicontinuous functions" applicable to integro-partial differential equations, NoDEA 13
(2006) 137-165.

\bibitem{Ki09}
Y. Kifer, Large deviations and adiabatic transitions for dynamical systems and Markov processes in fully coupled averaging, American Mathematical Soc., 2009.

\bibitem{Ki04}
Y. Kifer, Averaging principle for fully coupled dynamical systems and large deviations, Ergod. Theory Dyn. Syst. 24 (2004) 847-871.

\bibitem{KT81}
S. Karlin, H.E. Taylor, A second course in stochastic processes, Elsevier, 1981.

\bibitem{KP17}
R. Kumar, L. Popovic, Large deviations for multi-scale jump-diffusion processes, Stoch. Process. Their Appl. 127 (2017) 1297-1320.

\bibitem{KS12}
V. Knopova, R.L. Schilling, Transition density estimates for a class of L\'evy and L\'evy-type processes, J. Theor. Probab. 25 (2012) 144-170.

\bibitem{Li96}
R. Liptser, Large deviations for two scaled diffusions, Probab. Theory Relat. Fields 106 (1996) 71-104.

\bibitem{MW06}
A. Majda, X. Wang, Nonlinear dynamics and statistical theories for basic geophysical flows, Cambridge University Press, 2006.

\bibitem{Ph98}
H. Pham, Optimal stopping of controlled jump diffusion processes: a viscosity solution approach, JMSEC (1998).

\bibitem{Pa01}
T.N. Palmer, A nonlinear dynamical perspective on model error: A proposal for non-local stochastic-dynamic parametrization in weather and climate prediction models, Q. J. R. Meteorol. Soc. 127 (2001) 279-304.

\bibitem{Qi01}
H. Qian, Mesoscopic nonequilibrium thermodynamics of single macromolecules and dynamic entropy-energy compensation, Phys. Rev. E 65 (2001) 016102.

\bibitem{S15}
R.L. Schilling, An Introduction to L\'evy and Feller Processes, Dresden, 2015.

\bibitem{Sp13}
K. Spiliopoulos, Large deviations and importance sampling for systems of slow-fast motion, Appl. Math. Opt. 67 (2013) 123-161.

\bibitem{WRD}
W. Wang, A.J. Roberts, J. Duan, Large deviations and approximations for slow-fast stochastic reaction-diffusion equations, J.  Differ. Equations 253(2012) 3501-3522.

\bibitem{YD19}
S. Yuan, J. Duan, Action Functionals for Stochastic Differential Equations with L\'evy Noise, COSA 13 (2019) 10.
\end{thebibliography}
\end{document}